\documentclass[12pt]{article}
 \usepackage[T1]{fontenc}
 \usepackage[utf8]{inputenc}
 \usepackage{authblk}
 \usepackage[misc]{ifsym}
 \usepackage{dsfont}
 \usepackage{mathrsfs}
 \usepackage{amsmath,amssymb}
 \usepackage{amsfonts}
 
 \usepackage{amsthm}
 \usepackage{graphicx}
 \usepackage{subfigure}
 \usepackage{xcolor}
 
 \usepackage{bbm}
 \usepackage{mathrsfs}
 \usepackage{txfonts}
 \usepackage{color}
 \usepackage{pict2e}
 \usepackage{float}
 \usepackage{palatino,epsfig,latexsym}
 \usepackage{epsf,xy,epic,amscd}
 \usepackage{lineno}

 \usepackage{bibentry}
 \usepackage[colorlinks,linkcolor=blue,anchorcolor=blue,citecolor=blue,]{hyperref}

 \theoremstyle{plain}
 \newtheorem{theorem}{Theorem}[section]
 \newtheorem{lemma}[theorem]{Lemma}
 
 \newtheorem{proposition}[theorem]{Proposition}
 \newtheorem{corollary}[theorem]{Corollary}

 \newtheorem{problem}[theorem]{Problem}
 \newtheorem{Bounded Diameter Lemma}[theorem]{Bounded Diameter Lemma}
 \theoremstyle{definition}
 \newtheorem{definition}[theorem]{Definition}
 
 \newtheorem{remark}[theorem]{Remark}

 \newcommand{\bi}{\bar{i}}
 \newcommand{\bj}{\bar{j}}

 \newcommand{\Hmm}[1]{\leavevmode{\marginpar{\tiny%
 			$\hbox to 0mm{\hspace*{-0.5mm}$\leftarrow$\hss}%
 			\vcenter{\vrule depth 0.1mm height 0.1mm width \the\marginparwidth}%
 			\hbox to
 			0mm{\hss$\rightarrow$\hspace*{-0.5mm}}$\\\relax\raggedright #1}}}

 \hfuzz=\maxdimen
 \DeclareFixedFont{\Acknowledgment}{OT1}{cmr}{bx}{n}{14pt}
\begin{document}
%\pagewiselinenumbers
\numberwithin{equation}{section}

%\title{On the partial uniform ellipticity and its applications II: level set version}
\title{On the level set version of partial uniform ellipticity and  applications}
\author{Ri-Rong Yuan\thanks{School of Mathematics, \,South China University of Technology, \,Guangzhou 510641, \,China \\ \indent \indent Email address:\,yuanrr@scut.edu.cn
		%\,rirongyuan@stu.xmu.edu.cn
}}
\date{}
%Key words: Partial uniform ellipticity; concave function
\maketitle

\begin{abstract}

%The article is a sequel to \cite{yuan2020conformal}.
	We derive level set version of partial uniform ellipticity for symmetric concave functions. This suggests an   
	effective %and vital 
approach to investigate second order fully nonlinear  equations of elliptic and parabolic type.
	% As applications, we study a class of complex fully nonlinear equations, Weingarten equations and real Hessian equations.
	% As a by-product, we confirm a key inequalities for deriving interior derive estimates for Weingarten equations.

%\medskip
%\noindent{\bf Mathematics Subject Classifications (2020):} 52C26, 51M10, 57M50.

\end{abstract}

 %\tableofcontents

%\setcounter{section}{-1}

 \section{Introduction}

% \subsection{Background}
 
 Let $f$ be a smooth symmetric function defined in an open symmetric convex cone  $\Gamma\subset\mathbb{R}^n$ containing the positive cone
 \[\Gamma_n=\left\{\lambda\in \mathbb{R}^n: \mbox{ each } \lambda_i>0\right\}\subseteq\Gamma\]
 with vertex at the origin  and with nonempty boundary
  $\partial \Gamma\neq \emptyset$. 
 %It is reasonable to assume $f$ is not a constant. 
The study of fully nonlinear equations of the form
 \begin{equation}  \label{equation-cns} \begin{aligned}
 		\,& F(D^2u):= f(\lambda(D^2 u))=\psi \mbox{ in } \Omega\subset\mathbb{R}^n %\nonumber
 		% \,& u=\varphi \mbox{ on } \partial \Omega
 \end{aligned} \end{equation}
starts form the pioneering work \cite{CNS3} of Caffarelli-Nirenberg-Spruck.
 %initiated the study of the Dirichlet problem for a class of fully nonlinear equations 
 %which are determined by $f$, of eigenvalues of Hessians. 
 Since then the equations of this type have been extensively studied in real and complex variables.
 The following two basic hypotheses are imposed in the literature:
 % \begin{equation}\label{concave} \begin{aligned}
 % f \mbox{ is  concave in } \Gamma.
 % \end{aligned} \end{equation}
  \begin{equation}\label{concave}\begin{aligned}
 	\mbox{$f$ is a concave function in $\Gamma$,}
 \end{aligned}\end{equation}
 %i.e.   $f$ is a concave function in $\Gamma$. 
  \begin{equation}
 	\label{elliptic}
 	\begin{aligned}
 		\,& f_i(\lambda):=\frac{\partial f}{\partial \lambda_{i}}(\lambda)> 0  \mbox{ in } \Gamma,\,&  \forall 1\leq i\leq n.
 	\end{aligned}
 \end{equation}
In some cases one may replace \eqref{elliptic} by a weaker condition
\begin{equation}
	\label{elliptic-weak}
	\begin{aligned}
		f_{i}(\lambda) \geq 0 \mbox{ in } \Gamma, \quad \forall 1\leq i\leq n.
		% \mbox{  } \sum_{i=1}^n f_i(\lambda)>0   \mbox{ in } \Gamma.
	\end{aligned}
\end{equation}

As is well known, the typical examples satisfying \eqref{concave}-\eqref{elliptic}  are as follows: 
 \[f(\lambda)= \sigma_k^{1/k}(\lambda) \mbox{ or } (\sigma_k/\sigma_l)^{1/(k-l)}(\lambda), \mbox{ } 1\leq l<k\leq n, \mbox{ } \Gamma=\Gamma_k\]
 where $\sigma_k$
 is the $k$-th elementary symmetric function. Here
  $$\Gamma_k=\left\{\lambda\in\mathbb{R}^n: \sigma_j(\lambda)>0, \mbox{  } \forall 1\leq j\leq k\right\}.$$
 
  The linearlized operator of \eqref{equation-cns}  at $u$ is given by
$\mathfrak{L}_uw=\frac{\partial F(D^2u)}{\partial u_{ij}} \cdot w_{ij}.$
 %Let's briefly describe the structure of $\mathfrak{L}_u$:  
 One can check that the eigenvalues of $\left(\frac{\partial F}{\partial u_{ij}}(D^2u)\right)$
 % the leading term of %the linearlized operator 
 % $\mathfrak{L}$ at $u$ 
 are precisely given by
 \[f_1(\lambda), \cdots, f_n(\lambda) \mbox{ for } \lambda= \lambda(D^2 u).\]
% Thus  \eqref{elliptic} ensures \eqref{equation-cns} is elliptic at $u\in C^2(\Omega)$ when $\lambda(D^2u)\in \Gamma$.
%
 %As it is well known the Laplace equations $$\Delta u=\psi,$$ corresponding to $f(\lambda)=\sum_{i=1}^n\lambda_i$, closely connect to various problems arising from geometry and analysis. For such $f$ one sees that for each $1\leq i\leq n$ and $\lambda\in\mathbb{R}^n$
 %	\begin{equation}	\begin{aligned}
 %		f_i(\lambda)=1. \nonumber
 %\end{aligned}\end{equation}
 %This means that the Laplace equations are of uniform ellipticity which plays central roles in the study of Laplace equations.
 %The Poisson equation $\Delta u=\psi$, corresponding to  $f(\lambda)=\sum_{i=1}^n\lambda_i$, connects to various subjects from geometry, topology and physics.
In particular, for the Poisson equation  corresponding to  $f(\lambda)=\sum_{i=1}^n\lambda_i$,  
 \[f_i(\lambda)\equiv1, \quad  \forall \lambda\in\mathbb{R}^n, \quad \forall 1\leq i\leq n. \]
 This means that \eqref{equation-cns} is uniformly elliptic.
 % which plays central roles in the study.
  %Comparing with the Poisson equation, 
  However, the fully nonlinear equations analogous to \eqref{equation-cns} fail to be uniformly elliptic in general, which causes various hard difficulties in the investigation, especially in proof of  
  \textit{a priori} (interior) estimates. Consequently,
 it is important to compare $f_i(\lambda)$ with $\sum_{j=1}^n f_j(\lambda)$.  
 This leads to the notion of partial uniform ellipticity.
 % devoting to compare $f_i(\lambda)$ with $\sum_{i=1}^n f_i(\lambda)$.
 \begin{definition} [Partial uniform ellipticity] \label{def-PUE}
 	Let $H$ be a symmetric nonempty subset of $\Gamma$.
 	We say that $f$ is of 
 	\textit{$\mathrm{m}$-uniform ellipticity} in $H$,
 	if \eqref{elliptic-weak} holds 
 	 and there exists a uniform positive constant $\vartheta$ such that 
 	for any $\lambda\in H$ with  $\lambda_1 \leq \cdots \leq\lambda_n$,
 	\begin{equation}
 		\label{PUE1}
 		\begin{aligned}
 			f_{{i}}(\lambda) \geq   \vartheta \sum_{j=1}^{n}f_j(\lambda)>0, \quad \forall 1\leq i\leq \mathrm{m}.
 			% \nonumber
 			%\mbox{  } \cdots, \mbox{  } f_{{{1+\kappa_{\Gamma}}}}(\lambda) \geq  \vartheta_{\Gamma} \sum_{j=1}^{n}f_j(\lambda).
 		\end{aligned}
 	\end{equation}
 	In particular, \textit{$n$-uniform ellipticity} is also called fully uniform ellipticity.  
 \end{definition}

%A result of Lin-Trudinger \cite{Lin1994Trudinger} yields  
%For $(f,\Gamma)=(\sigma_k^{1/k},\Gamma_k)$,
%Based on the specific properties of $\sigma_k^{1/k}$, Lin-Trudinger \cite{Lin1994Trudinger} verified directly that $\sigma_k^{1/k}$ is of  $(n+1-k)$-uniform ellipticity in $\Gamma_k$.
 %In \cite{Lin1994Trudinger} Lin-Trudinger proved that 
 %the $k$-th elementary symmetric function 
 %$\sigma_k$ is of $(n-k+1)$-uniform ellipticity in $\Gamma_k$. However, their method does not apply to general functions.
 %However, their method cannot apply to general functions.  For more general cases, 
 %the Problem \ref{problem1-yuan} was partially fixed by
 %It would be important to compute $\mathrm{m}$.
The author \cite{yuan-PUE-conformal}\renewcommand{\thefootnote}{\fnsymbol{footnote}}\footnote{The paper  \cite{yuan-PUE-conformal} is essentially extracted from %\cite{yuan2020conformal,yuan2021conformal}
	[arXiv:2011.08580] and [arXiv:2101.04947].}
%\footnote{It was reorganized as a part of \cite{yuan-PUE-conformal}.}
%
 %\renewcommand{\thefootnote}{\fnsymbol{footnote}}\footnote{text overlap with arXiv: ???.} 
 % partially settled general cases left open by Lin-Trudinger. 
 introduced an integer $\kappa_\Gamma$   for $\Gamma$ %where $f$ is defined 
 %as follows
 \begin{equation}
 	\label{kappa_Gamma}
 	\begin{aligned}
 		\kappa_{\Gamma}=\max\left\{k: (-\alpha_1,\cdots,-\alpha_k,\alpha_{k+1},\cdots, \alpha_n)\in \Gamma, \mbox{ where } \alpha_j>0, \mbox{ } \forall 1\leq j\leq n\right\}  %\nonumber
 \end{aligned}\end{equation}
%and bridged $\kappa_\Gamma$ with the partial uniform ellipticity of $f$. More precisely, the author  proved that if $f$ satisfies \eqref{concave} and \eqref{addistruc}, 
and proved that the concave symmetric functions satisfying 
\begin{equation}
	\label{addistruc}
	\begin{aligned}
		%\mbox{For any $\sigma<\sup_{\Gamma}f$ and } \lambda\in \Gamma \mbox{  we have } \lim_{t\rightarrow +\infty}f(t\lambda)>\sigma,
		\lim_{t\rightarrow +\infty}f(t\lambda)>f(\mu) \mbox{ for any } \lambda, \mbox{ }\mu\in \Gamma
	\end{aligned}
\end{equation}   
is exactly of $(\kappa_\Gamma+1)$-uniform ellipticity in $\Gamma$.
% (see Theorem \ref{yuan-k+1} below). 
More precisely, there exists a uniform positive constant depending only on $\Gamma$ such that for any $\lambda\in\Gamma$ with $\lambda_1 \leq \cdots \leq\lambda_n$,
\begin{equation}
	\label{yuan-k+1-equ}
	\begin{aligned}
		f_{{i}}(\lambda) \geq   \vartheta_{\Gamma} \sum_{j=1}^{n}f_j(\lambda), \quad \forall  1\leq i\leq 1+ \kappa_{\Gamma}. %\nonumber
		%\mbox{  } \cdots, \mbox{  } f_{{{1+\kappa_{\Gamma}}}}(\lambda) \geq  \vartheta_{\Gamma} \sum_{j=1}^{n}f_j(\lambda).
	\end{aligned}
\end{equation}
%See \cite[Remark 5.6]{yuan-PUE-conformal} for a choice of $\vartheta_\Gamma$.
 %As applications, the author studied some prescribed curvature problems on the conformal classes of complete metrics.
% In the proof the concavity of $f$ plays crucial roles.
In the case $(f,\Gamma)=(\sigma_k^{1/k},\Gamma_k)$,   \eqref{yuan-k+1-equ} was proved by Lin-Trudinger \cite{Lin1994Trudinger}.
Such a partial uniform ellipticity is relevant to %study of fully nonlinear equations.
%which apply to 
various partial differential 
equations of elliptic and parabolic type. 
A surprising consequence of the conclusion \eqref{yuan-k+1-equ} is that a type 2 cone means in some sense that the corresponding equations %subject to \eqref{addistruc}
 are  uniformly elliptic. 
% In some cases, it is required to extend  %equations with less restrictions.
   %level set version of 
  % partial uniform ellipticity to more general context.   %the setting of level set version.
 %
%However, there are   some restrictions to Theorem \ref{yuan-k+1}. For instance,  it is not applicable to the equations satisfying for  $K>0$,
 
However, the condition \eqref{addistruc} is not fulfilled in some situations. 
% thus \eqref{yuan-k+1-equ} does not apply.
%Theorem \ref{yuan-k+1} directly.
For instance, 
it does not allow
%when only the right-hand side of the equation is known, and
 %we are interested in 
 % the equation satisfies
\begin{equation}
	\label{condition-K1}
	\begin{aligned}
	\sum_{i=1}^n f_i(\lambda)\lambda_{i}\geq 
	-K\sum_{i=1}^n f_i(\lambda),  \mbox{ for some } K\geq 0, \quad
	\forall %\inf\psi\leq f(\lambda)\leq \sup\psi.
	 \underline{A}\leq f(\lambda)\leq \overline{A}. 
	\end{aligned}
\end{equation}
Such a condition includes 
 \begin{equation}
\label{related-condition1}
\sum_{i=1}^n f_i(\lambda)\lambda_{i}\geq 	\delta>0, \quad	\forall %\inf\psi\leq f(\lambda)\leq \sup\psi.
	 \underline{A}\leq f(\lambda)\leq \overline{A}
	 \end{equation} 
	 as a special case.
These two conditions
 appeared in the study of  
 %Weingarten  curvature equations and  Hessian type
   certain fully nonlinear equations from differential geometry, see e.g. \cite{CNS5,Guan2004Spruck,ShengUrbasWang-Duke,Trudinger90,Guan12a,GSS14} and the references therein.
We shall remark that % \eqref{related-condition1} is a special case of \eqref{condition-K1}, and 
  assumptions \eqref{related-condition1} and \eqref{condition-K1} 
 hold only on the range of the
given function.

%In an attempt to treat more general second order fully nonlinear equations of elliptic and parabolic type,
%without assuming \eqref{addistruc},
% with less restrictions,
Motivated by this and related topics, %from PDEs and differential geometry,
 it would be necessary to 
%answer the Problem \ref{problem1-yuan} in the level set version. 
derive the level set version of partial uniform ellipticity.

% \vspace{2mm}  The article is organized as follows.  In Section \ref{sec2} we prove level set version of partial uniform ellipticity.  In Section \ref{sec3} we derive interior estimates for first and second order derivatives for complex fully nonlinear  equations. In Section \ref{sec5} we briefly discuss Hessian type equations and Weingarten equation as applications of Theorem \ref{coro1.6}. In Appendixes, we summarize some useful lemmas.
 %This note is a sequel to  \cite{yuan2020conformal}.

 % \vspace{1mm}
 %The paper is organized as follows.  In Section \ref{sec2} we prove some properties of $f$ and $\Gamma_{\sigma,f}$. In Section \ref{sec3} we complete the proof of main theorems. In Section \ref{sec4} Weingarten equations are briefly discussed. In Appendix \ref{appendix1} we present some characterizations of $f$ satisfying \eqref{addistruc}.

 %\section{Statement of results} 
 %\label{sec02}
 
 %The aim of this section is to set up level set version of partial uniform ellipticity. 

 %\section{Partial uniform ellipticity: Level set version} 

 %This article is a sequel to \cite{yuan2020conformal}. The level set version of partial uniform ellipticity is further studied below. 
% Our results apply to various partial differential equations of elliptic and parabolic type.
 % to Weingarten equations and real Hessian type fully nonlinear equations.

 %\subsection{Set-up}
 %\vspace{1mm}
 Before stating results we introduce some notions and impose 
 %certain   
 appropriate assumptions.
 For $\sigma$, we % have the common notation
 denote the level set by
 \[\partial\Gamma^\sigma=\{\lambda\in\Gamma: f(\lambda)=\sigma\}.\]
 %\mbox{  } \Gamma^\sigma=\{\lambda\in\Gamma: f(\lambda)>\sigma\}.\]
 Conditions \eqref{concave} and \eqref{elliptic-weak}   imply that  (see Lemma \ref{k-buchong2}) 
 \begin{equation}
 \label{sumfi-02}
 \sum_{i=1}^n f_i(\lambda)>0  \mbox{ for }
 f(\lambda)<\sup_\Gamma f.\end{equation}
 %(here we also use $f$ is not a constant),
 As a result, the level set $\partial \Gamma^\sigma$ 
 (when $\partial\Gamma^\sigma\neq\emptyset$) is a smooth complete noncompact convex hypersurface. 
 %contained in $\Gamma$.
 Throughout this paper we assume
 \[\sigma<\sup_{\Gamma}f \mbox{ and } \partial\Gamma^\sigma\neq\emptyset.\]
  Let's denote 
  \[t_\lambda=\frac{\sum_{i=1}^n f_i(\lambda)\lambda_i}{\sum_{j=1}^n f_j(\lambda)},\quad 
  \vec{\bf 1}=(1,\cdots,1)\in\mathbb{R}^n.\] 
 Geometrically, the tangent plane $T_{\lambda}\partial\Gamma^\sigma$ of $\partial \Gamma^\sigma$ at $\lambda\in\partial\Gamma^\sigma$,   intersects the diagonal at $t_\lambda \vec{\bf 1}$,   i.e.
 $$T_\lambda\partial\Gamma^\sigma\cap\{t\vec{\bf1}: t\in\mathbb{R}\}=\{t_\lambda\vec{\bf1}\}.$$

 We assume $t_\lambda$ has lower bound
 \begin{equation}
 	\label{key1}
 	\begin{aligned}
 		\liminf_{|\lambda|\rightarrow+\infty, \lambda\in\partial \Gamma^\sigma}  t_\lambda >-\infty.
 	\end{aligned}
 \end{equation}
 And then we denote 
 $$\tau_\sigma=\inf_{ \lambda\in\partial \Gamma^\sigma}  t_\lambda.$$

 Let $c_\sigma$ be the positive constant with 
 $f(c_\sigma \vec{\bf 1})=\sigma$.
 % (see Lemma \ref{k-buchong1}).
 %For $\lambda\in\partial\Gamma^\sigma$,   
 By \eqref{concave}, we know
 % $t_\lambda\leq c_\sigma$,
 % see e.g. \cite{yuan2021cvpde}. 
 \begin{equation}
 	%\label{tau-up1}
 	\tau_\sigma\leq c_\sigma  \nonumber
 \end{equation}
 with equality holding if and only if 
 \[f_1(\lambda)=f_2(\lambda)=\cdots =f_n(\lambda), %=\frac{1}{n}\sum_{i=1}^n f_i(\lambda) 
 \quad \forall \lambda\in\partial\Gamma^\sigma.\]
 Consequently, we assume throughout this paper that
 \begin{equation}
 	\label{012}
 	\begin{aligned}
 		\tau_\sigma<c_\sigma. \nonumber
 	\end{aligned}
 \end{equation}

 \begin{definition}
 	\label{def-k}
	Let $\Gamma_{\sigma,f}$ denote the cone 
 	$$
 	\Gamma_{\sigma,f}=  
 	\left\{t(\lambda-\tau_\sigma \vec{\bf 1}): \lambda\in\partial \Gamma^\sigma, t>0\right\}.$$
 	Furthermore $\Gamma_{\sigma,f}(1):=\partial\Gamma^\sigma-\tau_\sigma\vec{\bf 1}$ simply denotes a slice of $\Gamma_{\sigma,f}$.
 	For such $\Gamma_{\sigma,f}$, we define
 	$$\kappa_{\Gamma_{\sigma,f}}=\max \left\{k: (-\alpha_1,\cdots,-\alpha_k,\alpha_{k+1},\cdots,\alpha_n)\in \Gamma_{\sigma,f}, \mbox{  } \alpha_i>0\right\}.$$
 \end{definition}
 
 %Clearly, \[\kappa_{\Gamma_{\sigma,f}}=\max\{k: (-\alpha_1,\cdots,-\alpha_k,\alpha_{k+1},\cdots,\alpha_n)\in \Gamma_{\sigma,f}(1), \mbox{   } \alpha_i>0\}.\]

% \subsection{The results on partial uniform ellipticity}
%\subsection{Results on partial uniform ellipticity}
 %In this paper we prove level set version of partial uniform ellipticity  as follows. 
 
Below we state the results on partial uniform ellipticity.
 \begin{theorem} 
 	\label{mainthm-1}
 	Assume \eqref{concave}, \eqref{elliptic-weak} and \eqref{key1} hold.
 	Then there exists a uniform positive constant $\vartheta_{\Gamma_{\sigma,f}}$ depending only on $\Gamma_{\sigma,f}$ such that 
 	for each $ \lambda\in \partial\Gamma^\sigma$ with
 	% the order
 	$\lambda_1 \leq \cdots \leq\lambda_n$,
 	\begin{equation}
 		\begin{aligned}
 			f_{{i}}(\lambda) \geq   \vartheta_{\Gamma_{\sigma,f}} \sum_{j=1}^{n}f_j(\lambda), \quad \forall  1\leq i\leq 1+ \kappa_{\Gamma_{\sigma,f}}. \nonumber
 			%\mbox{  } \cdots, \mbox{  } f_{{{1+\kappa_{\Gamma}}}}(\lambda) \geq  \vartheta_{\Gamma} \sum_{j=1}^{n}f_j(\lambda).
 		\end{aligned}
 	\end{equation}
 	
 \end{theorem}

 \begin{remark}
    % In Lemma \ref{lemma24} we proved that if $f$, in addition, satisfies  
When replacing \eqref{elliptic-weak} and \eqref{key1}  by \eqref{addistruc}, 
 Theorem \ref{mainthm-1} gives back 
 %the following:
  % Theorem \ref{yuan-k+1}.
 \eqref{yuan-k+1-equ}.
 \end{remark}
As a consequence of Theorem \ref{mainthm-1},  we can confirm an important inequality.  
 \begin{theorem}
 	\label{coro1.6}
 	Suppose, in addition to 
 	\eqref{concave} and \eqref{elliptic-weak}, that  \eqref{condition-K1} holds for $\sup_{\partial\Gamma}f<\underline{A}< \overline{A}<\sup_\Gamma f$.
 	% and \eqref{key-123}.
 	% \begin{equation} \begin{aligned}
 	% \sum_{i=1}^n f_i(\lambda)\lambda_i\geq K_0\sum_{i=1}^n f_i(\lambda), \mbox{ } \forall \lambda\in \partial\Gamma^\sigma 
 	% \end{aligned} \end{equation} for some $K_0\in \mathbb{R}$. 
	Then there is a positive constant $\theta$ depending on %$\sigma$ %and 
	$K$, $\underline{A}$, $\overline{A}$
 	such that    	\begin{equation}
 		\label{key2-yuan-2}
 		\begin{aligned}
 			f_i(\lambda)\geq \theta \left(1+ \sum_{j=1}^n f_j(\lambda)\right) \,\mbox{ if } \lambda_i\leq -K,\mbox{ } \forall \underline{A}\leq f(\lambda)\leq \overline{A}.
 		\end{aligned}
 	\end{equation}
 	In particular, if $K=0$, i.e. $\sum_{i=1}^n f_i(\lambda)\lambda_i\geq0$ for $\underline{A}\leq f(\lambda)\leq \overline{A}$, then
 	 	\begin{equation}
 		\label{key2-yuan}
 		\begin{aligned}
 			f_i(\lambda)\geq \theta \left(1+ \sum_{j=1}^n f_j(\lambda)\right) \mbox{ if } \lambda_i\leq 0.
 		\end{aligned}
 	\end{equation}
 \end{theorem}

 \begin{remark}
 		The inequality \eqref{key2-yuan}
 		% or \eqref{key13-yuan} 
 		was imposed as a vital assumption by Li \cite{LiYY1991} and later by many experts %mathematicians
 		to study certain geometric PDEs from classical differential geometry and conformal geometry, see e.g.  \cite{Trudinger90,Guan1991Spruck,ShengUrbasWang-Duke,SChen2007,SChen2009,Guan1999CVPDE,Urbas2002,Guan12a}. 
		Our results can improve related results obtained there.
		%including Weingarten equations, 
 	 
 %	The inequality \eqref{key2-yuan} or \eqref{key13-yuan} was imposed as a key assumption by Li \cite{LiYY1991}, and subsequently by  
 %and later by 
 %Trudinger \cite{Trudinger90}, Guan-Spruck \cite{Guan1991Spruck}, Sheng-Urbas-Wang \cite{ShengUrbasWang-Duke} to study Weingarten equations.  Also it has more applications to real Hessian fully nonlinear equations and certain geometric PDEs from conformal geometry, see e.g.  \cite{SChen2007,SChen2009,Urbas2002,Guan12a,GSS14}.
 \end{remark}

  %The present paper 
  %is a sequel to \cite{yuan2020conformal}, which 
  The paper is organized as follows.
  The level set version of partial uniform ellipticity is derived in Section \ref{sec2}.
  As applications, we derive the interior estimates for first and second order derivatives for complex fully nonlinear equations with Laplacian terms in Section \ref{sec3}, and briefly discuss %the Weingarten equations and 
  real Hessian fully nonlinear equations  in Section \ref{sec5}.  In Appendixes, we summarize and prove some lemmas.

% \section{Some inequalities of symmetric concave functions}
% \label{sec2}

 \section{Partial uniform ellipticity: Level set version} 
  \label{sec2} 
  
   \subsection{Proof of Theorem \ref{mainthm-1}}
 The concavity assumption \eqref{concave} gives
 \begin{equation}\label{concave-1}\begin{aligned}
 		\sum_{i=1}^n f_i(\lambda)(\mu_i-\lambda_i)\geq f(\mu)-f(\lambda), \quad \forall \lambda, \mbox{ }\mu\in\Gamma. 
 \end{aligned}\end{equation}
 First we prove the following lemma. 
 \begin{lemma}
 	\label{lemma3.4-2}
 	Suppose \eqref{concave}, \eqref{elliptic-weak} and \eqref{key1} hold. 
 	Then 
 	\[\sum_{i=1}^n f_i(\lambda)\mu_i\geq0, \quad \forall \lambda\in\partial\Gamma^\sigma, \mbox{  } \forall \mu\in \Gamma_{\sigma,f}.\]
 \end{lemma}
 
 \begin{proof}
 	Given $\lambda\in\partial\Gamma^\sigma$.
 	Without loss of generality, we choose $\mu\in \Gamma_{\sigma,f}(1)$.
 	So $\mu+\tau_\sigma\vec{\bf 1}\in \partial\Gamma^\sigma$.
 	%The concavity of $f$ %(i.e. \eqref{concave}) %implies 
 	%\begin{equation} \label{concave} \begin{aligned}
 	%\,& \sum_{i=1}^n f_i(\lambda)(\mu_i-\lambda_i) \geq f(\mu)-f(\lambda), \,& \forall \lambda, \mbox{  } \mu\in\Gamma.
 	% \end{aligned} \end{equation}
 	The inequality \eqref{concave-1} simply yields
 	\begin{equation}
 		\label{key-12}
 		\begin{aligned}
 			\sum_{i=1}^n f_i(\lambda)(\tau_\sigma+\mu_i-\lambda_i)\geq 0.  \nonumber
 		\end{aligned}
 	\end{equation}
 	%It follows from \eqref{key-12} that 
 	Thus
 	$$\sum_{i=1}^n f_i(\lambda)\mu_i \geq \sum_{i=1}^n f_i(\lambda)\lambda_i-\tau_{\sigma}\sum_{i=1}^n f_i(\lambda)\geq0.$$ 
 	% whenever $|\lambda|\gg1$. 
 	% Together with  \eqref{key1}, 
 	%We complete the proof.
 	
 \end{proof}
 Let $\lambda_1\leq\cdots\leq\lambda_n$,   the concavity and symmetry of $f$ imply
 $$f_1(\lambda)\geq\cdots\geq f_n(\lambda) \mbox{ and } f_1(\lambda)\geq \frac{1}{n}\sum_{i=1}^n f_i(\lambda).$$
 If $\kappa_{\Gamma_{\sigma,f}}=0$ then Theorem \ref{mainthm-1}  clearly follows. For  $\kappa_{\Gamma_{\sigma,f}}\geq1$, 
 it is a consequence of the following proposition.
 \begin{proposition}
 	\label{yuanrr-2}
 	In addition to \eqref{concave}, \eqref{elliptic-weak} and \eqref{key1}, 
 	we assume $\kappa_{\Gamma_{\sigma,f}}\geq1$.
 	Let
 	$\alpha_1, \cdots, \alpha_n$ be $n$ positive constants with
 	$$(-\alpha_1,\cdots,-\alpha_{\kappa_{\Gamma_{\sigma,f}}}, \alpha_{1+\kappa_{\Gamma_{\sigma,f}}},\cdots, \alpha_n)\in \Gamma_{\sigma,f}.$$
 	Assume in addition that $\alpha_1\geq\cdots\geq \alpha_{\kappa_{\Gamma_{\sigma,f}}}$.
 	Then for  $ \lambda\in \partial\Gamma^\sigma$ with $\lambda_1 \leq \cdots \leq\lambda_n$,
 	\begin{equation}\label{theta1}
 		\begin{aligned}
 			f_{1+\kappa_{\Gamma_{\sigma,f}}}(\lambda)\geq\frac{\alpha_1}{\sum_{i=1+\kappa_{\Gamma_{\sigma,f}}}^n
 				\alpha_i-\sum_{i=2}^{\kappa_{\Gamma_{\sigma,f}}}\alpha_i}f_1(\lambda).
 		\end{aligned}
 	\end{equation}
 	%Furthermore, $f_1(\lambda)\geq \frac{1}{n}\sum_{i=1}^n f_i(\lambda).$
 \end{proposition}
 
 \begin{proof}
 	%The proof follows the outline in \cite{yuan2020conformal}. %For convenience we present the proof here.
 	According to Lemma \ref{lemma3.4-2}, we have
 	\begin{equation}
 		\label{good1-yuan}
 		\begin{aligned}
 			-\sum_{i=1}^{\kappa_{\Gamma_{\sigma,f}}} \alpha_i f_i(\lambda)+\sum_{i=1+\kappa_{\Gamma_{\sigma,f}}}^n \alpha_i f_i(\lambda) \geq0. \nonumber
 		\end{aligned}
 	\end{equation}
 	This simply yields $f_{1+\kappa_{\Gamma_{\sigma,f}}}(\lambda)
 	\geq\frac{\alpha_1}{\sum_{i=1+\kappa_{\Gamma_{\sigma,f}}}^n \alpha_i}f_1(\lambda)$.
 	In addition, 
 	one derives \eqref{theta1} by using iteration. %and \eqref{good1-yuan}.
 	%This completes the proof.
 	
 \end{proof}
 
\begin{remark}
	In the case $\kappa_{\Gamma_{\sigma,f}}\geq1$, the constant 
 $\vartheta_{\Gamma_{\sigma,f}}$  in Theorem \ref{mainthm-1} can be achieved as
	\[\vartheta_{\Gamma_{\sigma,f}}=\sup_{(-\alpha_1,\cdots,-\alpha_{\kappa_{\Gamma_{\sigma,f}}}, \alpha_{1+\kappa_{\Gamma_{\sigma,f}}},\cdots, \alpha_n)\in \Gamma_{\sigma,f}}\frac{\alpha_1/n}{\sum_{i=1+\kappa_{\Gamma_{\sigma,f}}}^n
		\alpha_i-\sum_{i=2}^{\kappa_{\Gamma_{\sigma,f}}}\alpha_i}.\]
\end{remark}
 
 %\begin{remark}
 %In fact we can prove 
 % \begin{equation} \begin{aligned}
 % f_{1+\kappa_{\Gamma_{\sigma,f}}}(\lambda) \geq \frac{\alpha_1/n+\varrho_\sigma(\lambda)}{\sum_{i=1+\kappa_{\Gamma_{\sigma,f}}}^n  \alpha_i-\sum_{i=2}^{\kappa_{\Gamma_{\sigma,f}}}\alpha_i}\sum_{i=1}^n f_i(\lambda), \nonumber
 %\end{aligned}\end{equation}
 %where $\varrho_\sigma(\lambda)=t(t_\lambda-\tau_\sigma)$ for some $t>0$.
 %\end{remark}
 
 %\begin{remark}
 %a similar theorem analogous to
 %From  the proof above, Theorem \ref{yuan-k+1} can be extended to the case when $f$ satisfies  \eqref{concave} and \eqref{addistruc}.
 %{remark}

  \subsection{A new criterion for $f$ satisfying \eqref{addistruc}}
%\subsection*{More lemma}

 Building on Lemma %\ref{lemma-new-2} and
 \ref{lemma3.4}, we can deduce a new criterion for \eqref{addistruc}.
 \begin{lemma}
 	\label{lemma24}
 	In the presence of \eqref{concave},   \eqref{elliptic-weak}  and 
 	\begin{equation}
 		\label{sumfi>0}
 		\begin{aligned}
 			\sum_{i=1}^n f_i(\lambda)>0 \mbox{ in } \Gamma, 
 		\end{aligned}
 	\end{equation} 
 condition  \eqref{addistruc} 
 	is equivalent to 
 	\begin{equation}
 		\label{addistruc24}
 		\begin{aligned}
 		\Gamma\subseteq \Gamma_{\sigma,f}, \quad \forall \sup_{\partial \Gamma}f<\sigma<\sup_\Gamma f.
 		\end{aligned}
 	\end{equation} 
 \end{lemma}
 
 \begin{proof}
 	
 	$\Leftarrow$ Fix $\lambda$, $\mu\in\Gamma$, let $\sigma=f(\lambda)$.
 	%By \eqref{addistruc24} we have
 	%\[\mu=t_0(\tilde{\mu}-\tau_\sigma\vec{\bf1}) \mbox{ for some } t_0>0, \tilde{\mu}\in\partial\Gamma^\sigma,\]
 	%which gives $\sum_{i=1}^n f_i(\lambda)\mu_i\geq 0$.
 	Since $\Gamma\subseteq\Gamma_{\sigma,f}$, 
 	we have $\sum_{i=1}^n f_i(\lambda)\mu_i\geq 0$ by Lemma \ref{lemma3.4-2}.  Thus
  \eqref{addistruc} holds by Lemma  \ref{lemma-new-2}.

 	$\Rightarrow$ For $\lambda\in\Gamma$ and $\sup_{\partial\Gamma}f<\sigma<\sup_{\Gamma}f$, one has $f(t\lambda)>\sigma$ for some $t>0$. By Lemma \ref{lemma3.4}, $\tau_\sigma\geq0$. There is $0<t_0<t$ such that $f(t_0\lambda+\tau_\sigma\vec{\bf1})=\sigma$.  
 	This yields
 	$$\Gamma\subseteq\Gamma_{\sigma,f}.$$
 	
 \end{proof}

 %Theorem \ref{yuan-k+1} 
 The $(\kappa_\Gamma+1)$-uniform ellipticity as asserted in \eqref{yuan-k+1-equ} follows as a consequence of Theorem \ref{mainthm-1}, Lemma \ref{lemma24} and Corollary \ref{coro3.2}. 
 % Here we also use  Corollary \ref{coro3.2}.
 %Furthermore, condition \eqref{elliptic-weak} is  satisfied automatically according to  Corollary \ref{coro3.2}. 

 \subsection{Confirming an inequality}
 %\subsection{Proof of Theorem \ref{coro1.6}}
 %As a consequence of Theorem \ref{mainthm-1}, we first 
% We deduce the following proposition to verify an interesting inequality.
 \begin{proposition}
 	\label{mainthm-2}
 	%Let $f$ be as in Theorem \ref{mainthm-1},  
 	Suppose  \eqref{concave}, \eqref{elliptic-weak} and \eqref{key1} hold.
 	Then for $\lambda\in\partial\Gamma^\sigma$, 
 	\begin{equation}
 		\label{345}
 		\begin{aligned}
 			f_i(\lambda)\geq \vartheta_{\Gamma_{\sigma,f}}\sum_{j=1}^n f_j(\lambda) \mbox{ whenever } \lambda_i\leq \tau_\sigma.
 		\end{aligned}
 	\end{equation} 
 	%where $\vartheta_{\Gamma_{\sigma,f}}$ is the constant from Theorem \ref{mainthm-1}. 
 	In particular,  replacing \eqref{key1} by %\eqref{key-123},
 	\begin{equation}
 		\label{key-123}
 		\sum_{i=1}^n f_i(\lambda)\lambda_i\geq 0 \mbox{ in } \partial\Gamma^\sigma,
 	\end{equation}
 	then for any $\lambda\in\partial\Gamma^\sigma$ we get
 	\begin{equation}
 		\label{key1-yuan}
 		\begin{aligned}
 			f_i(\lambda)\geq  \vartheta_{\Gamma_{\sigma,f}}\sum_{j=1}^n f_j(\lambda)  \mbox{ if } \lambda_i\leq0. %\nonumber
 		\end{aligned}
 	\end{equation}
 \end{proposition}

 \begin{proof}%[Proof of Theorem \ref{mainthm-2}]
 	
 	Given $\lambda\in\partial\Gamma^\sigma$ with $\lambda_1\leq \cdots\leq \lambda_n$. Let $\mu_{i}=\lambda_i-\tau_\sigma$, then 
 	$\mu=(\mu_1,\cdots,\mu_n)\in \Gamma_{\sigma,f}$. %and $\mu_1\leq\cdots\leq \mu_n$.
 	By the definition of $\kappa_{\Gamma_{\sigma,f}}$, $\mu_{1+\kappa_{\Gamma_{\sigma,f}}}\geq0$, i.e. $\lambda_{1+\kappa_{\Gamma_{\sigma,f}}}\geq\tau_\sigma$.
 	%Suppose now $\lambda_i<\tau_\sigma$, i.e. $\mu_i<0$. So the index  $$i\leq \kappa_{\Gamma_{\sigma,f}},$$ then \eqref{345} follows as a consequence of Theorem \ref{mainthm-1}
 	For each $\lambda_i\leq\tau_\sigma$, we have
 %	$\lambda_i\leq \lambda_{1+\kappa_{\Gamma_{\sigma,f}}}$, and then
 	$$f_i(\lambda)\geq f_{1+\kappa_{\Gamma_{\sigma,f}}}(\lambda).$$
 	% If $\mu_{1+\kappa_{\Gamma_{\sigma,f}}}>0$ then the index  $$i\leq \kappa_{\Gamma_{\sigma,f}}.$$
 	% If $\mu_{1+\kappa_{\Gamma_{\sigma,f}}}=0$ and $\mu_i=0$ then $f_i(\lambda)=f_{1+\kappa_{\Gamma_{\sigma,f}}}(\lambda)$.
 	Consequently, \eqref{345} follows from Theorem \ref{mainthm-1}.
 \end{proof}

 It follows from \eqref{concave-1} that if $f$ satisfies  \eqref{key1} then there is a positive constant $\theta=\theta(\sigma)$
 depending only on $\sigma$ 
 %and $\tau_\sigma$ 
 such that 
 \begin{equation}\label{sumfi-2} \begin{aligned}
 		\sum_{i=1}^n f_i(\lambda)\geq \theta(\sigma)\mbox{ in }  \partial\Gamma^\sigma. %\nonumber
 \end{aligned} \end{equation}
 Theorem \ref{coro1.6} then follows from \eqref{sumfi-2} and Proposition \ref{mainthm-2}. 
 
 %%%%%%%%%%%%%%%%
 \vspace{1mm}
Below we consider two special cases.

 %Next, we are going to confirm such an inequality when $f$ satisfies \eqref{t1-to-0}.
 %according to  Theorem \ref{coro1.6}, Lemmas \ref{lemma3.4} and \ref{lemma1-con-addi}.
 \begin{lemma}
 	 	\label{yuan-lemma1-weingarten}
 	If $(f,\Gamma)$ satisfies 
 	\eqref{concave}, \eqref{elliptic} and
 \begin{equation}
	\label{t1-to-0}
	\begin{aligned}
		\lim_{t\rightarrow0^+}f(t\vec{\bf1})>-\infty,
	\end{aligned}
\end{equation}
then we have $\sum_{i=1}^n f_i(\lambda)\lambda_i>0$ and \eqref{key2-yuan}.
 \end{lemma}
 \begin{proof}
 	According to Lemmas \ref{lemma3.4} and \ref{lemma1-con-addi}, we have $\sum_{i=1}^n f_i(\lambda)\lambda_i>0$. Theorem \ref{coro1.6} then gives \eqref{key2-yuan}.
 \end{proof}

  \begin{lemma}
 	\label{yuan-lemma2-weingarten}
 	Let $(f,\Gamma)$ satisfy \eqref{concave} and %\eqref{positive-1},
 	\begin{equation}
 		\label{positive-1}
 		\begin{aligned}
 			\,& f>0 \mbox{ in } \Gamma, \,& f=0 \mbox{ on } \partial \Gamma 
 		\end{aligned}
 	\end{equation} 
 	then we have $\sum_{i=1}^n f_i(\lambda)\lambda_i\geq0$,  \eqref{elliptic-weak} and \eqref{key2-yuan}. 
 \end{lemma}
 \begin{proof}
 	The conclusions \eqref{elliptic-weak} and $\sum_{i=1}^n f_i(\lambda)\lambda_i\geq0$ are deduced from Lemma  \ref{lemma-new-1}.
 	Again, by Theorem \ref{coro1.6} we have \eqref{key2-yuan}.
 \end{proof}
 
\begin{remark}
	These two lemmas allows one to improve some results on Weingarten equations obtained by Li \cite{LiYY1991} and Trudinger \cite{Trudinger90} respectively;  
	we decide to omit the details here.
	% to cut the paper's length.
\end{remark}
 %coincides with 
 %%%%%%%%%%%%%%%%%%%%%
  
 %\section{Applications to certain geometric fully nonlinear PDEs}
 
 \section{Applications to complex fully nonlinear equations}
 \label{sec3}
 
 Let $(M,\omega)$ be a compact Hermitian manifold of complex dimension $n\geq 2$ possibly with boundary. Let $\chi$ be a smooth real $(1,1)$-form, $\psi$  a $C^2$-smooth function, $\Delta$  the Laplacian operator,  
 $$ Z = \frac{1}{(n-1)!}*\mathfrak{Re}(\sqrt{-1}\partial u\wedge \overline{\partial}\omega^{n-2}),$$ where $*$ is the Hodge star operator with respect to $\omega$.
 Recently, Sz\'ekelyhidi-Tosatti-Weinkove \cite{STW17} proved 
 Gauduchon's conjecture, by
 %thereby extending the Calabi-Yau theorem to Gauduchon metric. 
 solving the Monge-Amp\`ere equation for $(n-1)$-PSH functions on a closed Hermitian manifold 
% $(M,\omega)$
 \begin{equation}
 	\label{MA-n-1}
 	\begin{aligned}
 		\left(\omega_0+\frac{1}{n-1}\left(\Delta u\omega-\sqrt{-1}\partial \overline{\partial}u\right)+Z\right)^n =e^{(n-1)\phi}\omega^n, \quad  \omega_0>0.
 	\end{aligned}
 \end{equation}
 When $M$ admits a balanced metric and an astheno-K\"ahler metric, it closely connects to the Form-type Calabi-Yau equation
   \cite{FuWangWuFormtype2010}.
   %which is proposed by Fu-Wang-Wu \cite{FuWangWuFormtype2010} to extend the Calabi-Yau theorem to balanced metric. 
 %The main difficulty in  Sz\'ekelyhidi-Tosatti-Weinkove's proof is to deal with the bad terms due to the gradient terms $\partial u$, $\overline{\partial} u$ from $Z$.
 Subsequently, 
 %assuming existence of subsolutions,
  the author \cite{yuan2021-3}\renewcommand{\thefootnote}{\fnsymbol{footnote}}\footnote{It is essentially extracted from the second parts of  %\cite{yuan2019,yuan2021}.
[arXiv:2001.09238] and [arXiv:2106.14837].} %(overlapped with \cite{yuan2021-3})
 solved the Dirichlet problem for \eqref{MA-n-1}, thereby extending Sz\'ekelyhidi-Tosatti-Weinkove's results to complex manifolds with boundary.
 In addition, the author \cite{yuan2021-3} has investigated
 the Dirichlet problem for  equations of the form %fully nonlinear elliptic equations  
 \begin{equation}
 	\label{n-1-equ1}
 	\begin{aligned}
 		f \left(\lambda\left(\chi+\frac{1}{n-1}\left(\Delta u\omega-\sqrt{-1}\partial \overline{\partial}u\right)+Z\right)\right) =\psi
 	\end{aligned}
 \end{equation}
% which have the form  $$F(A)=f(\lambda(A)) $$
 where $\chi$ is a smooth real $(1,1)$-form, $\lambda(A)$ are the eigenvalues of $A$ 
 %(a real $(1,1)$-form) 
 with respect to $\omega$.   
 %The study of fully nonlinear elliptic equations depending linearly on gradient terms attack attention  since then, see for example  \cite{yuan2018CJM,yuan2019}. 
More precisely, 
when imposing \eqref{addistruc}, $\Gamma=\Gamma_n$
 and the unbound condition  
 \begin{equation}
 	\label{unbound}
 	\begin{aligned}
 		\lim_{t\to +\infty}f(\lambda_1+t,\cdots,\lambda_{n-1}+t, \lambda_n)= \sup_\Gamma f, \quad \forall \lambda=(\lambda_1,\cdots,\lambda_n)\in\Gamma,
 	\end{aligned}
 \end{equation} 
the Dirichlet problem for \eqref{n-1-equ1} was solved by the author in \cite[Section 7]{yuan2021-3}.
The assumption \eqref{unbound} allows
\[f=({\sigma_n}/{\sigma_{k}})^{1/(n-k)}, \quad 0\leq k\leq n-2.\] 
 %As a contrast, 
 While $\Gamma\neq\Gamma_n$, even  without assuming \eqref{unbound} as well with degenerate right-hand side,
 the Dirichlet problem for \eqref{n-1-equ1}  was completely solved there.
  %in Part 2 of \cite{yuan2019}.
%  in \cite[Section 7]{yuan2021-3}.
  These results reveal that there are significant differences between 
 $\Gamma=\Gamma_n$ and   $\Gamma\neq \Gamma_n$. 
 %The differences  between %such two cases $\Gamma=\Gamma_n$ and $\Gamma\neq \Gamma_n$ are our motivation.
Based on the partial uniform ellipticity, we are able to figure out those differences. 
 \begin{proposition}
 	\label{example1-yuan}
 	 Given $(f,\Gamma)$, we define 
 	 \begin{equation}
 	 	\label{n-1-equation}
 	 	\begin{aligned}
 	 		\tilde{f}(\lambda)={f}(\mu), \quad \mu_i=\sum_{j\neq i}\lambda_j, \quad  \tilde{\Gamma}=\{\lambda\in\mathbb{R}^n: \mu\in \Gamma\}.
 	 		%\mbox{ where } 
 	 		%\mu=\left(\sum_{j=1}^n\lambda_j \right)\vec{\bf 1}-\lambda, \mbox{ } \vec{\bf 1}=(1,\cdots,1).\]
 	 		%\mu_i=\sum_{j\neq i}\lambda_j.
 	 	\end{aligned}
 	 \end{equation}
 	%  $\widetilde{\Gamma}=\left\{\lambda:  \mu=(\mu_1,\cdots,\mu_n)\in\Gamma \mbox{ for }   \mu_i=\sum_{j\neq i}\lambda_{j}\right\}$ and $\widetilde{f}(\lambda)=f(\mu)$. 
 	  %$\mu_i=\sum_{j\neq i}\lambda_{j}$ 
 	In the presence of  \eqref{concave} and \eqref{addistruc}, we have the following:
 	\begin{itemize}
 		\item[$\mathrm{\bf (1)}$] If $\Gamma\neq\Gamma_n$
 	  holds, then $\tilde{\Gamma}$ is of type 2 cone and $\tilde{f}$ 
 		is of fully uniform ellipticity in $\tilde{\Gamma}$.
 	 \item[$\mathrm{\bf (2)}$] If $\Gamma=\Gamma_n$, then $\tilde{f}$ 
 	 is of $(n-1)$-uniform ellipticity in $\tilde{\Gamma}$.
 	\end{itemize}
 
 \end{proposition}

 The proposition shows that in the presence of \eqref{concave} and \eqref{addistruc}, equation \eqref{n-1-equ1} is   uniformly elliptic in the case $\Gamma\neq\Gamma_n$, while it is only of $(n-1)$-uniform  ellipticity when $\Gamma=\Gamma_n$.
 Such differences 
 %between %such two cases$\Gamma=\Gamma_n$ and $\Gamma\neq \Gamma_n$ 
% indicates that  \eqref{n-1-equ1} shall have certain surprising properties when $\Gamma\neq\Gamma_n$.
  motivate us to derive more delicate results.
 % if one in addition assumes  $\Gamma\neq \Gamma_n$. 
 %As an application of partial uniform ellipticity,
We apply level set version of partial uniform ellipticity to  derive interior estimates for \eqref{n-1-equ1},
 when imposing proper restrictions: %to $(f,\Gamma)$:  
 \begin{equation}
 	\label{neqgamma_n}
 	\begin{aligned}
 		\Gamma\neq \Gamma_n,
 	\end{aligned}
 \end{equation} 
 \begin{equation}
 	\label{con10}
 	\begin{aligned}	
 		\lim_{t\to +\infty} f(t,\cdots,t,0)>\sup_M\psi,
 	\end{aligned}
 \end{equation}
 \begin{equation}
 	\label{con11}
 	\begin{aligned}
 		\sum_{i=1}^n f_i(\lambda)\lambda_i\geq 0  \mbox{ in }
 		\Gamma^{\underline{\psi},\overline{\psi}} 
 	\end{aligned}
 \end{equation} 
 where 
 \begin{equation}
 	\label{con11-2}
 	\begin{aligned}
 		\Gamma^{\underline{\psi},\overline{\psi}}  =\left\{\lambda: \inf_M\psi\leq f(\lambda)\leq \sup_M\psi\right\}.
 	\end{aligned}
 \end{equation}
 
 %\begin{remark} Conditions \eqref{con10} and \eqref{con11} are broader than \eqref{addistruc}. \end{remark}

 %Such equations have various applications to non-K\"ahler geometry.  For $f=(\sigma_n)^{1/n}$, equation \eqref{n-1-equ1} is the Monge-Amp\`ere equation for $(n-1)$-PSH functions which arises from Gauduchon's conjecture.   

 \begin{theorem}
 	\label{thm1-inter}
 	Let $B_r$ be a geodesic ball in $(M,\omega)$. Suppose \eqref{concave}, \eqref{elliptic-weak}, \eqref{neqgamma_n}, \eqref{con10} and \eqref{con11} hold. Then for any solution $u\in C^4(B_r)$ to   equation \eqref{n-1-equ1}  satisfying
 	\begin{equation}
 		\label{admissible-f1}
 		\begin{aligned}
 			\lambda\left(\chi+\frac{1}{n-1}\left(\Delta u\omega-\sqrt{-1}\partial \overline{\partial}u\right)+Z\right)\in \Gamma 
 			\mbox{ in } B_r, % \nonumber
 		\end{aligned}
 	\end{equation}
 	there is a uniform positive constant $C$ depending only on $|u|_{C^0(B_r)}$, $|\psi|_{C^2(B_r)}$ and geometric quantities on $B_r$,
 	such that 
 	\begin{equation}
 		\label{inter-estimate}
 		\begin{aligned}
 			\sup_{B_{r/2}} (|\nabla u|^2+|\partial\overline{\partial}u|)\leq \frac{C}{r^2}.  %\nonumber
 		\end{aligned}
 	\end{equation}
 \end{theorem}
 
 %As described by a topological obstruction presented in  \cite{yuan2020conformal}, it is believed that condition $\Gamma\neq\Gamma_n$ is a crucial assumption and can not be dropped.
 %As a contrast, the case $\Gamma=\Gamma_n$ is rather different.
 % since \eqref{assumption-4} does not allow the critical case $\varrho=1$ when $\Gamma=\Gamma_n$. 
 
 %Condition \eqref{con10} is weaker than \eqref{addistruc}.
 According to Lemma \ref{lemma3.4} and Corollary \ref{coro3.2} below,   \eqref{elliptic-weak} and \eqref{con11} are simultaneously satisfied when $f$ satisfies \eqref{concave} and \eqref{addistruc}. As a consequence,
 we obtain
 \begin{corollary}
 	\label{Coro1.7}
 	Theorem \ref{thm1-inter} holds when $(f,\Gamma)$ satisfies \eqref{concave},  \eqref{addistruc} and \eqref{neqgamma_n}.
 \end{corollary}
 
 %As shown in \cite{yuan2020conformal}, conditions \eqref{elliptic-weak}, \eqref{key-123} and \eqref{unbound-2} are satisfied by $(f,\Gamma)$ when it satisfies \eqref{addistruc} and $\Gamma\neq\Gamma_n$.
 % which indicates that the case $\Gamma\neq\Gamma_n$ is relatively standard.  
 %This observation can be applied to study conformal deformation of the Einstein tensor \cite{yuan2020conformal,yuan2021conformal}, also to extend the main results of \cite{GQY2018}. 
 
 In fact interior estimate \eqref{inter-estimate} still holds %(with the same proof)
  for %fully nonlinear 
 equations
 \begin{equation}
 	\label{n-1-equation1}
 	\begin{aligned}
 		f\left(\lambda\left(\chi+\Delta u \omega -\varrho \sqrt{-1}\partial\overline{\partial} u+\gamma Z\right)\right)=\psi  %\nonumber
 	\end{aligned}
 \end{equation} 
 provided that $(f,\Gamma)$ satisfies \eqref{concave} and \eqref{addistruc}, $\gamma$ is a $C^2$-smooth function, and 
 $\varrho$   %as in \eqref{n-1-equation1} 
 is a $C^2$-smooth function satisfying
 %an appropriate arrangement
 \begin{equation}
 	\label{assumption-4}
 	\begin{aligned}
 		%\Gamma\neq \Gamma_n  (\mbox{i.e. }  \kappa_\Gamma\geq 1) \mbox{ and }
 		\varrho<\frac{1}{1-\kappa_\Gamma \vartheta_{\Gamma}} \mbox{ and } \varrho\neq 0.
 		%\mbox{  where $\vartheta_{\Gamma}$ is as in Theorem \ref{yuan-k+1}.}
 	\end{aligned}
 \end{equation}
Here $\kappa_\Gamma$ and $\theta_\Gamma$ are the constants in \eqref{kappa_Gamma} and \eqref{yuan-k+1-equ}, respectively.
 %Theorem \ref{yuan-k+1}. 
 %Note that \eqref{assumption-4} does not allow the critical case $\varrho=1$ when $\Gamma=\Gamma_n$. 
 % More precisely, we obtain
 \begin{theorem}
 	\label{thm23-yuan}
 	Let $B_r$ be a geodesic ball in $(M,\omega)$, and let $u\in C^4(B_r)$ be a solution to \eqref{n-1-equation1} in $B_r$ with 
 	\begin{equation}
 		\begin{aligned}
 			\lambda\left(\chi+\Delta u \omega-\varrho \sqrt{-1}\partial\overline{\partial} u+\gamma Z\right)\in\Gamma
 			\mbox{ in } B_r.  \nonumber
 		\end{aligned}
 	\end{equation} 
 	Suppose in addition that \eqref{concave}, \eqref{addistruc} and \eqref{assumption-4} hold. Then 
 	\begin{equation}
 		%\label{inter-estimate}
 		\begin{aligned}
 			\sup_{B_{r/2}} (|\nabla u|^2+|\partial\overline{\partial}u|)\leq \frac{C}{r^2}    \nonumber
 		\end{aligned}
 	\end{equation}
 	where $C$ depends only on $|u|_{C^0(B_r)}$, $|\psi|_{C^2(B_r)}$ and geometric quantities on $B_r$.
 \end{theorem}
 
 The restriction \eqref{assumption-4} to parameter $\varrho$ 
 was imposed by the author %\cite{yuan2020conformal}
  \cite{yuan-PUE-conformal} to study conformal deformation of modified Schouten tensors. When $\Gamma=\Gamma_n$  %(notice $\kappa_{\Gamma_n}=0$),
  it reduces to
 \begin{equation}
 	\label{rhosmall1}
 	\varrho<1 \mbox{ and } \varrho\neq 0. 
 \end{equation}
Hence it does not cover the Monge-Amp\`ere equation for $(n-1)$-PSH function.
 %\eqref{MA-n-1}.
 %Under the assumption of \eqref{rhosmall1},
The equations analogous to \eqref{n-1-equation1} were also
 studied in \cite{GQY2018}  %where the method depends essentially on
 under the assumption \eqref{rhosmall1}.  
 %It is noteworthy that 
Notice \eqref{assumption-4} allows the critical case $\varrho=1$ when $\Gamma\neq\Gamma_n$. Our result in Theorem \ref{thm23-yuan} is new.
 
 \begin{remark}
 	The main difficulty in  Sz\'ekelyhidi-Tosatti-Weinkove's proof of second estimate is to deal with the bad terms due to the gradient terms $\partial u$, $\overline{\partial} u$ from $Z$. Their proof depends heavily on   delicate structures of $Z$, which cannot be extended to more general cases. In contrast with Sz\'ekelyhidi-Tosatti-Weinkove's estimates, 
	our results assert the interior estimates for second and first order derivatives when $\Gamma\neq\Gamma_n$. 
 	In fact, such interior estimates are not true for general complex fully nonlinear equations. 
 \end{remark}

 \subsection{Interior estimates for  equations of fully uniform ellipticity}

% Let $(M,\omega)$ be a  Hermitian manifold % possibly with boundary,  of complex dimension $n\geq 2$, with K\"ahler form  $\omega=\sqrt{-1}g_{i\bar j} dz_i\wedge d\bar z_j$. Let $\chi=\sqrt{-1} \chi_{i\bar j} dz_i\wedge d\bar z_j$ be a smooth real $(1,1)$-form.
 
 In this subsection we are concerned with an equation
 of the form
 \begin{equation}
 	\label{main-equ2}
 	\begin{aligned}
 		%f(\chi+\sqrt{-1}\partial\overline{\partial} u +)
 		F(\mathfrak{g}_{i\bar j}):=f(\lambda(\mathfrak{g}_{i\bar j}))=\psi 
 	\end{aligned}
 \end{equation}
where $\mathfrak{g}_{i\bar j}=u_{i\bar j}+\chi_{i\bar j}+ S_{i\bar j}^k u_k+\overline{S_{j\bar i}^k}u_{\bar k}$,   $\lambda(\mathfrak{g}_{i\bar j})$ denote the eigenvalues of $\mathfrak{g}_{i\bar j}$ with respect to $g_{i\bar j}$. 
We call $u$ an \textit{admissible} function if $\lambda(\mathfrak{g}_{i\bar j})\in \Gamma$.
In addition, we assume that 
%$f$ is of fully uniform ellipticity. Namely, 
there exists a positive constant  $\theta$ such that
\begin{equation}
	\label{assum-fue1}
	\begin{aligned}
		f_i(\lambda) \geq \theta \sum_{j=1}^n f_j(\lambda)   \mbox{ in } \Gamma^{\underline{\psi},\overline{\psi}}, \quad
		\forall 1\leq i\leq n.
	\end{aligned}
\end{equation}
%\begin{equation}
%	\label{sum-2}	\begin{aligned}
%	  \sum_{i=1}^n f_i(\lambda) \geq \kappa \quad \mbox{ in } \Gamma^{\underline{\psi},\overline{\psi}},
%	\end{aligned}\end{equation}
%where $\Gamma^{\underline{\psi},\overline{\psi}}$ is given by \eqref{con11-2}.
% Following \cite{CNS3}, $u$ is \textit{admissible} if $\lambda(\mathfrak{g}_{i\bar j})\in \Gamma$. 

  \vspace{1.5mm}
  We prove the following interior  estimates.
 \begin{theorem}
 	\label{thm1-inter-3}
 	Let $B_r$ be a geodesic ball in $(M,\omega)$. Suppose   \eqref{concave}, \eqref{elliptic-weak}, \eqref{con10} and \eqref{assum-fue1} hold. Then for any admissible solution $u\in C^4(B_r)$ to   \eqref{main-equ2} in $B_r$, we have 
 	\begin{equation}
 		\begin{aligned}
 			\sup_{B_{r/2}} (|\partial\overline{\partial} u|+|\nabla u|^2)\leq \frac{C}{r^2}  \nonumber
 		\end{aligned}
 	\end{equation}
 	where $C$ is a uniform constant depending only on
 	$\theta^{-1}$, $|u|_{C^0(B_r)}$,  $|\psi|_{C^2(B_r)}$ and geometric quantities 
	on $B_r$.
 \end{theorem}
 
 % Besides the work of Sz\`ekylyhidi-Tosatti-Weinkove \cite{STW17}, there are some other works concerning complex fully nonlinear equations with gradient terms, see e.g. \cite{TW-pamq,yuan2018CJM,yuan2019,yuan-pamq}. As shown there, the gradient terms contained in equations cause serious difficulties for deriving second estimate. As a contrast, we can derive interior second estimate for \eqref{main-equ2}.

 \subsubsection{Preliminaries}

 %\subsubsection{Linearized operator}
 The linearized operator of equation \eqref{main-equ2}, say  $\mathcal{L}$, at solution $u$ is locally given by
 \begin{equation}
 	\mathcal{L} v=F^{i\bar j} v_{i\bar j} + F^{i\bar j}S_{i\bar j}^k v_k + F^{i\bar j} \overline{S_{j\bar i}^k} v_{\bar k} 
 \end{equation}
 where $F^{i\bar j}=\frac{\partial F}{\partial \mathfrak{g}_{i\bar j}}$. 
 One knows that the eigenvalues of $F^{i\bar j}$ (w.r.t. $g^{i\bar j}$) 
 are precisely $$f_1(\lambda), \cdots, f_n(\lambda), \quad \mbox{ where } \lambda=\lambda(\mathfrak{g}).$$ 
 Moreover
 \[F^{i\bar j}\mathfrak{g}_{i\bar j}=\sum_{i=1}^n f_i(\lambda)\lambda_i, \quad \sum F^{i\bar j}g_{i\bar j}=\sum_{i=1}^n f_i(\lambda).\]
 From condition \eqref{con10}, the right-hand side must satisfy
 \[\sup_M\psi<\sup_\Gamma f.\]
 Combining with Lemma \ref{k-buchong2} below, 
 \begin{equation}
 	\begin{aligned}
 		\sum_{i,j=1}^n F^{i\bar j} {g}_{i\bar j}=\sum_{i=1}^n f_i(\lambda)>0.
 	\end{aligned}
 \end{equation}
 
 \begin{remark}
 The condition \eqref{assum-fue1} yields that equation \eqref{main-equ2} is in effect uniformly elliptic 
 at {admissible} solution $u$. 
 That is automatically satisfied when imposed the same conditions as that of Theorem \ref{thm1-inter} or Corollary \ref{Coro1.7}. 
 \end{remark}
 
% \subsubsection{Some useful formulas}
 The following formulas are standard 
 \begin{equation}
 	\label{deco1}
 	\begin{aligned}
 		u_{1\bar j k}-u_{k\bar j 1}  = \,& T^l_{1k}u_{l\bar j},  \\   
 		%\end{aligned}\end{equation}
 		%\begin{equation}\begin{aligned}
 		u_{1\bar 1 i\bar i}-u_{i\bar i 1\bar 1} =R_{i\bar i 1\bar p}u_{p\bar 1}-
 		R_{1\bar 1 i\bar p}u_{p\bar i} \,& +2\mathfrak{Re}\{\bar T^{j}_{1i}u_{i\bar j 1}\}+T^{p}_{i1}\bar T^{q}_{i1}u_{p\bar q}.   
 	\end{aligned}
 \end{equation}

 Denote  \[w = |\nabla u|^2 \mbox{ and } Q=|\partial\overline{\partial}u|^2+|\partial \partial u|^2.\]
 Under local coordinates $z=(z_1,\cdots,z_n)$ around  $z_0$, with $g_{i\bar j}(z_0)=\delta_{ij}$, we have 
 by straightforward computations
 \[w_i = u_{\bar k} u_{ki} + u_{k} u_{i\bar k},\]
 %\begin{equation} \begin{aligned}
 	\[	w_{i\bar j} 
 		%= \,& u_{\bar k} u_{ki\bar j}  + u_{ki} u_{\bar k\bar j} 	+ u_{k\bar j} u_{i\bar k} + u_{k} u_{i\bar k\bar j} \\
 		=  u_{ki} u_{\bar k\bar j} + u_{k\bar j} u_{i\bar k} 
 		+ u_{\bar k} u_{i\bar j k} + u_{k} u_{i\bar j \bar k} 
 		+ R_{i\bar j k\bar l} u_{\bar k} u_l
 		- T^{l}_{ik} u_{l\bar j} u_{\bar k} - \overline{T^{l}_{jk}} u_{i\bar l} u_k. \] %\nonumber
 %\end{aligned} \end{equation}
 One then obtains 
 \begin{lemma}
 	\label{lemma-Lw}
 	We have
 	\[F^{i\bar j}w_i w_{\bar j} \leq 2 wQ\sum F^{i\bar i};\]
 	and there exists $C>0$ such that
 	\begin{equation}
 		\begin{aligned}
 			\mathcal{L}(w) \geq   \frac{3\theta Q}{4} \sum F^{ii}-Cw\sum F^{i\bar i}-C|\nabla \psi| \sqrt{w}. \nonumber
 		\end{aligned}
 	\end{equation}
 	
 \end{lemma}

 \subsubsection{Interior estimate for first derivative}

 Let's consider the quantity 
 \[m_{0} = \max_{\bar M} \eta  |\nabla u|^2 e^{\phi} \]
 where $\eta \geq 0$ and $\phi = \phi (z, u)$ are functions 
 to be determined.  
 Suppose that $m_{0}$ is attained at an interior point $z_0 \in M$. 
 %By Lemma 2.9 in \cite{ST11} we may 
 We choose local coordinates $(z_{1}, \ldots, z_{n})$ such that 
 $g_{i\bar j} = \delta_{ij}$ 
 %and $T_{ij}^k = 2 \Gamma_{ij}^k$
 at $z_0$.
 As above we denote $w = |\nabla u|^2$. Without loss of generality, $w\geq 1$ at $z_0$.  
 From above, the function 
 $\log \eta+\log w +\phi$ 
 achieves a maximum at $z_{0}$ and therefore, 
 \begin{equation}
 	\label{gqy-G32}
 	\frac{\eta_{i}}{\eta}+\frac{w_{i}}{w}+\phi_{i} = 0, \;\;
 	\frac{\eta_{\bar i}}{\eta}+\frac{w_{\bi}}{w}+\phi_{\bar i} = 0,
 \end{equation}
 %and 
 \begin{equation}
 	\label{gqy-G34}
 	%F^{i\bj} \Big(	\frac{\eta_{i\bj}}{\eta} 	- \frac{\eta_i \eta_{\bj}}{\eta^{2}} + \frac{w_{i\bj}}{w} 	- \frac{w_i w_{\bj}}{w^{2}} + \phi_{i\bj}\Big) \leq 0.
 	\mathcal{L}(\log \eta+\log w+\phi)\leq 0.
 \end{equation}
 
 Combining \eqref{gqy-G32} with Cauchy-Schwarz inequality, we derive
 \begin{equation}
 	\frac{1}{w^2} F^{i\bar j}w_i w_{\bar j} \leq \frac{1+\epsilon}{\epsilon \eta^2} F^{i\bar j}\eta_i \eta_{\bar j}
 	+(1+\epsilon) F^{i\bar j} \phi_i \phi_{\bar j}.
 \end{equation}
 As a result, combining with Lemma \ref{lemma-Lw} and let  $8\epsilon\leq \theta$, we derive at $z_0$
 \begin{equation}
 	\label{L-logw}
 	\begin{aligned}
 		\mathcal{L} \log w %=\,& \frac{\mathcal{L}w}{w}-F^{i\bar j}\frac{w_i w_{\bar j}}{w^2} \\
 		\geq \,& 
 		\frac{(3\theta-8\epsilon)  Q}{4w} \sum F^{i\bar i}
 		-C\sum F^{i\bar i} -\frac{C|\nabla \psi|}{\sqrt{w}} 
 		\\\,&
 		-\frac{1-\epsilon^2}{\epsilon \eta^2} F^{i\bar j} \eta_i \eta_{\bar j}
 		-(1-\epsilon^2) F^{i\bar j} \phi_i \phi_{\bar j} 
 			\\	\geq \,& 
 		\frac{  \theta Q}{2w} \sum F^{i\bar i}	-C\sum F^{i\bar i} -\frac{C|\nabla \psi|}{\sqrt{w}} 
		%\\\,&
			 -\frac{1}{\epsilon \eta^2} F^{i\bar j} \eta_i \eta_{\bar j}	 -  F^{i\bar j} \phi_i \phi_{\bar j}. 
 	\end{aligned}
 \end{equation}
 On the other hand, 
 \begin{equation}
 	\label{L-logeta}
 	\begin{aligned}
 		\mathcal{L} \log \eta 
 		%=\,& \frac{\mathcal{L}\eta}{\eta}-F^{i\bar j}\frac{\eta_i \eta_{\bar j}}{\eta^2} \\
 		= \frac{F^{i\bar j}\eta_{i \bar j}}{\eta}
 		+F^{i\bar j} S_{i\bar j}^k \frac{\eta_k}{\eta} +F^{i\bar j} \overline{S_{j\bar i}^k} \frac{\eta_{\bar k}}{\eta}
 		-F^{i\bar j}\frac{\eta_i \eta_{\bar j}}{\eta^2}.
 	\end{aligned}
 \end{equation}
 
 To derive the interior estimate,
 following \cite{Guan2003Wang} (see also \cite{GQY2018}) we take $\eta$ to be a smooth 
 function with compact support in $B_{r}\subset M$
 satisfying
 \begin{equation}
 	\label{2-22}
 	0\leq \eta \leq 1, ~~\eta|_{B_{\frac{r}{2}}}\equiv 1,
 	~~|\nabla \eta | \leq \frac{C \sqrt{\eta}}{r},  
 	%\equiv \frac{C \zeta^{1 - \alpha}}{r \alpha},
 	~~|\nabla^{2} \eta| \leq \frac{C}{r^2}.
 	%\equiv \frac{C}{r \alpha^2}.
 \end{equation}
Thus
 \begin{equation}
 	\label{2-23}
 	\begin{aligned}
 		\frac{1+\epsilon}{\epsilon \eta^{2}} F^{i\bj} \eta_i \eta_{\bj} 
 		+F^{i\bar j} S_{i\bar j}^k \frac{\eta_k}{\eta} +F^{i\bar j} \overline{S_{j\bar i}^k} \frac{\eta_{\bar k}}{\eta}
 		- \frac{1}{\eta} F^{i\bj} \eta_{i\bj}  
 		\leq  \frac{C}{\epsilon r^2 \eta} \sum F^{i\bi}.
 	\end{aligned}
 \end{equation}

 As in \cite{Guan2008IMRN},  
 %(see also \cite{GQY2018,yuan2020conformal}), 
 let $\phi=v^{-N}$ where $v=u-\inf_{B_r} u+2$ ($v\geq 2$ in $B_r$)  and $N\geq 1$ is an integer that is chosen later. By direct computation 
 \begin{equation}
 	\label{phii}
 	\begin{aligned}
 		\,&  \phi_i=-Nv^{-N-1}  u_i,  \quad \phi_{\bar i}=-Nv^{-N-1} u_{\bar i}
 		\\
 		\,& \phi_{i\bar j} 
 		=N(N+1)v^{-N-2} u_i u_{\bar j}-Nv^{-N-1} u_{i\bar j}.   \nonumber
 	\end{aligned}
 \end{equation}
 %which % couple with fully uniform ellipticity \eqref{fully-uniform2}, 
 So
 \begin{equation}
 	\label{5-12}
 	\begin{aligned}
 		F^{i\bar i} \phi_i \phi_{\bar j}=N^2 v^{-2N-2} F^{i\bar j} u_iu_{\bar j}
 	\end{aligned}
 \end{equation}
 and
 \begin{equation}
 	\label{Lphi}
 	\begin{aligned}
 		\mathcal{L}\phi
 		= \,& N(N+1)v^{-N-2}F^{i\bar j} u_i u_{\bar j} -Nv^{-N-1}F^{i\bar j} u_{i\bar j}  
 		  -Nv^{-N-1}(F^{ij}S_{i\bar j}^k u_k + F^{ij} \overline{S_{j\bar i}^k} 
 		u_{\bar k})
 		\\
 		=\,& N(N+1)v^{-N-2}F^{i\bar j} u_i u_{\bar j} -Nv^{-N-1}F^{i\bar j}    (\mathfrak{g}_{i\bar j}-\chi_{i\bar j}).
 	\end{aligned}
 \end{equation} 
 The concavity of $f$ implies that there is a positive constant $C_1$ such that
 \begin{equation}
 	\label{5-14}
 	F^{i\bar j}\mathfrak{g}_{i\bar j}=\sum_{i=1}^n f_i \lambda_i \leq C_1 \sum_{i=1}^n f_i.
 \end{equation}
 %see also  %\eqref{tau-up1}. 
 %\eqref{012}.
 
 We choose $N\gg1$ so that $Nv^{-N}<1$, then 
 \begin{equation}\label{5-15}
 	N(N+1)v^{-N-2}-N^2 v^{-2N-2} \geq N^2 v^{-N-2}.
 \end{equation}
 
 Plugging \eqref{L-logw}-\eqref{L-logeta} and
 % \eqref{2-23}, \eqref{5-12}, \eqref{Lphi},  \eqref{5-14} and \eqref{5-15} 
 \eqref{2-23}-\eqref{5-15} 
 into  \eqref{gqy-G34}, we obtain
 \begin{equation}
 	\begin{aligned}
 		\theta w N^2v^{-N-2} \sum F^{i\bar i}+\frac{\theta Q}{2w}\sum F^{i\bar i} 
 		\leq CN v^{-N-1} \sum F^{i\bar i}+ \frac{C}{r^2 \eta }\sum F^{i\bar i}
 		+\frac{C}{\sqrt{w}}. \nonumber
 	\end{aligned}
 \end{equation}
 Note  
 \[\frac{\theta w N^2v^{-N-2}}{2} \sum F^{i\bar i}+\frac{\theta Q}{2w}\sum F^{i\bar i} \geq \theta N v^{-\frac{N}{2}-1} \sqrt{Q}\sum F^{i\bar i}\]
 and that there exists $R_0>0$ such that for any $\lambda$ with $|\lambda|\geq R_0$
 \[|\lambda|\sum_{i=1}^n f_i(\lambda) \geq \frac{f(|\lambda|\vec{\bf 1})-f(\lambda)}{2}
 %=\frac{f(|\lambda|\vec{\bf 1})-\psi}{2}
 \geq \frac{f(R_0\vec{\bf 1})-\psi}{2}>0.\]
 
 As a result, we obtain Theorem \ref{thm1-inter-3}.

 \subsubsection{Interior estimate for second order derivatives}
  
We derive  second order interior estimate.
% in Theorem~\ref{thm1-inter-3}. 

 \begin{proof}
 	As in \cite{GQY2018} 
 	we  consider the following
 	quantity  %which is given in local coordinates
 	\begin{equation}
 		\begin{aligned}
 			P:= \sup_{z \in M} \max_{\xi \in T^{1,0}_z M}
 			\; e^{2 \phi} \mathfrak{g}_{p\bar q} \xi_p \bar{\xi_q}
 			\sqrt{g^{k\bar l}  \mathfrak{g}_{i\bar l}  \mathfrak{g}_{k\bar j}  \xi_i \bar{\xi_j}}/|\xi|^3  \nonumber
 		\end{aligned}
 	\end{equation}
 	where $\phi$ is a function depending on $z$ and $|\nabla u|$.
 	Assume that it is achieved at an interior point $p_0 \in M$ for some
 	$\xi \in T^{1,0}_{p_0} M$. %See also related references cited in \cite{GQY2018}.
 	%As in the previous section we choose local coordinates around $z_0$  such that
 	%$g_{i\bj} = \delta_{ij}$ and $T_{ij}^k = 2 \Gamma_{ij}^k$
 	%using the lemma of Streets and Tian~\cite{ST11}, and that $\fg_{i\bj}$ is diagonal
 	%at $z_0$ with
 	The quantity $P$  is inspired by \cite{Tosatti2013Weinkove}.
 	We choose local coordinates $z=(z_{1},\cdots,z_{n})$ around $p_0$, such that at $p_0$
 	$$g_{i\bar j}=\delta_{ij}, \quad \mathfrak{g}_{i\bar j}=\delta_{ij}\lambda_{i} \mbox{ and } F^{i\bar j}=\delta_{ij}f_i.$$
 	%As in \cite{Tosatti2013Weinkove,GQY2018}  the maximum $A$ is achieved for $\xi = \partial_1$ at $p_0$.
 	The maximum $P$ is achieved for $\xi = \partial_1$ at $p_0$.
 	We assume $\mathfrak{g}_{1\bar{1}} \geq 1$;  otherwise we are done.
 	
 	In what follows the computations are given at $p_0$.
 	Similar to the computations in \cite{GQY2018} one has
 	\begin{equation}
 		\label{mp1}
 		\mathfrak{g}_{1\bar{1} i} + \mathfrak{g}_{1\bar{1}} \phi_i = 0, \;\;
 		\mathfrak{g}_{1\bar{1}\bar i} +  \mathfrak{g}_{1\bar{1}} \phi_{\bar i} = 0,
 	\end{equation}
 	\begin{equation}
 		\label{gblq-C90}
 		\begin{aligned}
 			0  \geq   \frac{F^{i\bar i}   \mathfrak{g}_{1\bar{1} i\bar i}}{\mathfrak{g}_{1\bar 1}}
 			+ F^{i\bar i} (\phi_{i\bar i} - \phi_i \phi_{\bar i})
 			% \\ &
 			+ \frac{1}{8  \mathfrak{g}_{1\bar{1}}^2} \sum_{k>1} F^{i\bar i}
 			\mathfrak{g}_{1 \bar k i}   \mathfrak{g}_{k\bar{1} \bar i}
 			- C\sum F^{i\bar i}.
 		\end{aligned}
 	\end{equation}
 	%(Notice that our notations in the proof is slightly different from that of \cite{STW17}).
 	
 	Combining with the standard formula \eqref{deco1}, together with straightforward  computation, we can derive 
 	\begin{equation}
 		\begin{aligned}
 			\mathfrak{g}_{1\bar 1 i\bar i} \geq \,&
 			\mathfrak{g}_{i\bar i 1\bar 1}+
 			2\mathfrak{Re} (\bar T^j_{1i}\mathfrak{g}_{1\bar j i})
 			+  2\mathfrak{Re} 
 			(S_{1\bar 1}^l\mathfrak{g}_{i\bar i l}-
 			S_{i\bar i}^l\mathfrak{g}_{1\bar 1 l})
 			-C\sqrt{Q} 
 		\end{aligned}
 	\end{equation}
 	where as above one denotes
 	\[Q=|\partial\overline{\partial}u|^2+|\partial \partial u|^2.\]

 	Differentiating equation \eqref{main-equ2}  twice 
 	(using covariant derivative), we obtain
 	\begin{equation}
 		\label{gqy-45diff the equation}
 		\begin{aligned}
 			F^{i\bar i}\mathfrak{g}_{i\bar i l} \,& =\psi_l, \\   \end{aligned}
 	\end{equation}
 	\begin{equation}
 		\begin{aligned}
 			F^{i\bar i}\mathfrak{g}_{i\bar i 1\bar 1}
 			=\psi_{1\bar 1}-F^{i\bar j,l\bar m }\mathfrak{g}_{i\bar j1}\mathfrak{g}_{l\bar m\bar 1}.   
 		\end{aligned}
 	\end{equation}
 	Then we have
 	\begin{equation}
 		\label{mp3}
 		\begin{aligned}
 			F^{i\bar i}\mathfrak{g}_{1\bar 1 i\bar i}
 			\geq \,&
 			2\mathfrak{Re} (F^{i\bar i}\bar T^j_{1i}\mathfrak{g}_{1\bar j i})
 			- 2\mathfrak{Re} 
 			F^{i\bar i}S_{i\bar i}^l\mathfrak{g}_{1\bar 1 l}
 			-C\sqrt{Q}\sum F^{i\bar i}.
 		\end{aligned}
 	\end{equation}
 	
 	Putting the above inequalities
 	into \eqref{gblq-C90} we get
 	\begin{equation}
 		\label{gblq-C90-2}
 		\begin{aligned}
 			0  \geq  \,&
 			\mathfrak{g}_{1\bar{1}}F^{i\bar i} (\phi_{i\bar i} - \phi_i \phi_{\bar i})
 			- 2\mathfrak{Re} 
 			(F^{i\bar i}S_{i\bar i}^l\mathfrak{g}_{1\bar 1 l})
 				+
 			2\mathfrak{Re} (F^{i\bar i}\bar T^1_{1i}\mathfrak{g}_{1\bar 1 i})
 			-C\sqrt{Q}\sum F^{i\bar i}
 			\\
 			%= \,&	\mathfrak{g}_{1\bar{1}} F^{i\bar i} (\phi_{i\bar i} - \phi_i \phi_{\bar i})	+2\mathfrak{g}_{1\bar 1}\mathfrak{Re} (F^{i\bar i}S_{i\bar i}^l\phi_l)	-2\mathfrak{g}_{1\bar 1}\mathfrak{Re} (F^{i\bar i}\bar T^1_{1i}\phi_{i})-C\sqrt{Q}\sum F^{i\bar i} \\
 			=\,&
 			\mathfrak{g}_{1\bar 1} \mathcal{L}\phi -\mathfrak{g}_{1\bar 1}F^{i\bar i}\phi_i\phi_{\bar i}
 			-2\mathfrak{g}_{1\bar 1}\mathfrak{Re} (F^{i\bar i}\bar T^1_{1i}\phi_{i})
 			-C\sqrt{Q}\sum F^{i\bar i}.  \nonumber
 		\end{aligned}
 	\end{equation}
 	%Here we also use \eqref{mp1} in deriving the second equality.

 	Let $\phi=\log\eta+\varphi(w)$, where as above $w = |\nabla u|^2$, and $\eta$ is the cutoff function given by  \eqref{2-22}. Then
 	\begin{equation}
 		\begin{aligned}
 			\mathcal{L}\phi=\,&
 			\frac{\mathcal{L}\eta}{\eta} -F^{i\bar i}\frac{|\eta_i|^2}{\eta^2}+\varphi' \mathcal{L}w +\varphi''F^{i\bar i}|w_i|^2,
 		\end{aligned}
 	\end{equation}
 	\begin{equation}
 		\begin{aligned}
 			F^{i\bar i}|\phi_i|^2
 			+2 \mathfrak{Re} F^{i\bar i}\bar T^1_{1i}\phi_{i} \leq \frac{4}{3}F^{i\bar i}|\phi_i|^2+C\sum F^{i\bar i}
 		\end{aligned}
 	\end{equation}
 	and  
 	\begin{equation}
 		\begin{aligned}
 			F^{i\bar i} |\phi_i|^2
 			\leq \frac{3}{2}F^{i\bar i}|\varphi_i|^2+3F^{i\bar i}\frac{|\eta_i|^2}{\eta^2}.  \nonumber
 		\end{aligned}
 	\end{equation}
 	
 As in \cite{Guan2008IMRN} we set $$\varphi=\varphi(w)=\left(1-\frac{w}{2N}\right)^{-\frac{1}{2}} \mbox{ where   $N=\sup_{\{\eta>0\}}|\nabla u|^2$}. $$
 	One can check $\varphi'=\frac{\varphi^3}{4N},$ $\varphi''=\frac{3\varphi^5}{16N^2}$ and $1\leq \varphi\leq \sqrt{2}$.
 	And so
 	\begin{equation}
 		\begin{aligned}
 			\varphi''-2\varphi'^2=\frac{\varphi^5}{16N^2}(3-2\varphi)>\frac{\varphi^5}{96N^2}.
 		\end{aligned}
 	\end{equation}
 	By Lemma \ref{lemma-Lw} we have
 	\begin{equation}
 		\begin{aligned}
 			\mathcal{L}(w) \geq   \frac{3\theta Q}{4} \sum F^{ii}-C (1+\sum F^{i\bar i}).
 		\end{aligned}
 	\end{equation}
 By \eqref{2-22} we obtain
 	\begin{equation}
 		\begin{aligned}
 			0\leq \eta \leq 1, ~~\eta|_{B_{\frac{r}{2}}}\equiv 1,
 			~~F^{i\bar i}\frac{|\eta_i|^2}{\eta^2} \leq \frac{C}{r^2 \eta},  
 			%\equiv \frac{C \zeta^{1 - \alpha}}{r \alpha},
 			~~\frac{\mathcal{L}\eta}{\eta} \leq  \frac{C}{r^2\eta} \sum F^{i\bar i}.
 			%\equiv \frac{C}{r \alpha^2}.
 		\end{aligned}
 	\end{equation}
 	
 	In conclusion we finally derive
 	\begin{equation}
 		\begin{aligned}
 			0
 			%\geq \,& \varphi' \mathcal{L}w +(\varphi''-2\varphi'^2)F^{i\bar i}|w_i|^2	+\frac{\mathcal{L}\eta}{\eta} -5F^{i\bar i}\frac{|\eta_i|^2}{\eta^2} 	-C\frac{\sqrt{Q}}{\mathfrak{g}_{1\bar{1}}} \sum F^{i\bar i} \\
 			%	\geq \,& 	\frac{3\varphi^3}{4N} \mathcal{L}w %%+\frac{\varphi^5}{96N^2}F^{i\bar i}|w_i|^2
 			%	+\frac{\mathcal{L}\eta}{\eta} -5F^{i\bar i}\frac{|\eta_i|^2}{\eta^2}-C\frac{\sqrt{Q}}{\mathfrak{g}_{1\bar{1}}} \sum F^{i\bar i}   \\
 			\geq %\,&
 			\frac{9\theta Q}{16N} \sum F^{i\bar i}-\frac{C}{r^2\eta}
 			\sum F^{i\bar i}
 			-C\frac{\sqrt{Q}}{\mathfrak{g}_{1\bar{1}}} \sum F^{i\bar i}. 
 		\end{aligned}
 	\end{equation}
 	This gives 
 	\[\eta \mathfrak{g}_{1\bar 1}\leq \frac{C}{r^2}.\]
 	%as required.
 \end{proof}

 %\section{Fully uniform ellipticity of $(n-1)$  equations}
 %\section{Interior estimates for equations of $(n-1)$-type}
 \subsection{Completion the proof of Theorem \ref{thm1-inter}}
% \label{sec4}
 %\subsection{Fully uniform ellipticity of certain operators}
 First, Theorem \ref{mainthm-1} implies the following result.
 \begin{theorem}
 	\label{thm-3}
 	Assume \eqref{concave}, \eqref{elliptic-weak} and \eqref{key-123} hold.  Then for fixed $\sigma$,
 	\begin{itemize}
 		
 		\item[$\mathrm{\bf (1)}$] If $ \partial\Gamma^\sigma\setminus\bar\Gamma_n \neq \emptyset,$
 		%\begin{equation}\label{key-234} \begin{aligned}
 		% \partial\Gamma^\sigma\setminus\bar\Gamma_n \neq \emptyset,
 		% \end{aligned} \end{equation}
 		%$\bar\Gamma_n=\{\lambda\in\mathbb{R}^n: \mbox{ each } \lambda_i\geq 0\}$.
 		then $\kappa_{\Gamma_{\sigma,f}}\geq 1$. 
 		%Here $\bar\Gamma_n$ is the closure of $\Gamma_n$.
 		
 		\item[$\mathrm{\bf (2)}$]
 		If $\Gamma$ is of type 2 cone %in the sense of \cite{CNS3} 
 		and %$f$ satisfies 
 		\begin{equation}
 			\label{unbound-1}
 			\begin{aligned}
 				% f(\lambda_1,\cdots,\lambda_{n-1},R)>\sigma, \mbox{  }  (\lambda_1,\cdots,\lambda_{n-1},R)\in\Gamma \mbox{ for some } R>0, 
 				f(0,\cdots,0,t)>\sigma \mbox{ for some } t>0,  %\nonumber
 			\end{aligned}
 		\end{equation} 
 	 then $\kappa_{\Gamma_{\sigma,f}}=n-1$  and $f$ is of fully uniform ellipticity when restricted to 
 		 $\partial\Gamma^\sigma$.
 	\end{itemize}
 \end{theorem}

\begin{proof}
Prove $\mathrm{\bf (1)}$: By \eqref{key-123}, $\tau_\sigma\geq 0$. Let $\lambda\in \partial\Gamma^\sigma\setminus\bar\Gamma_n$, and we assume $\lambda_n<0$. Note that $\lambda-\tau_\sigma \vec{\bf 1}\in \Gamma_{\sigma,f}$ and $\lambda_n-\tau_\sigma\leq \lambda_n<0$. Thus $\kappa_{\Gamma_{\sigma,f}}\geq1.$

Prove $\mathrm{\bf (2)}$: From \eqref{unbound-1}, there are  some positive constants $\epsilon_0$, $t_0$ such that $f(-\epsilon_0,\cdots,-\epsilon_0,t_0)>\sigma$. And then there is $0<\beta<1$ such that $f(-\beta\epsilon_0,\cdots,-\beta\epsilon_0,\beta t_0)=\sigma.$ Thus $(-\beta\epsilon_0-\tau_\sigma,\cdots,-\beta\epsilon_0-\tau_\sigma,\beta t_0-\tau_\sigma)\in \Gamma_{\sigma,f}$ as required. 
\end{proof}
 
 Next we present %satisfying all the assumptions imposed Theorem \ref{thm-3}. 
 functions being of fully uniform ellipticity.
% Let's denote  $\mu_i=\sum_{j\neq i}\lambda_j$ and
 %\begin{equation}	\label{n-1-equation}	\begin{aligned}
 	%	\widetilde{f}(\lambda)={f}(\mu), \quad  \widetilde{\Gamma}=\{\lambda\in\mathbb{R}^n: \mu\in \Gamma\}.
		 %\mbox{ where } 
 		%\mu=\left(\sum_{j=1}^n\lambda_j \right)\vec{\bf 1}-\lambda, \mbox{ } \vec{\bf 1}=(1,\cdots,1).\]
 		%\mu_i=\sum_{j\neq i}\lambda_j.
 %	\end{aligned} \end{equation}
% where $\mu_i=\sum_{j\neq i}\lambda_j.$
 
 \begin{corollary}
 	\label{coro-n-1-equ}
 	Let $(\tilde{f},\tilde{\Gamma})$ be as defined in \eqref{n-1-equation}. Let $\sup_{\partial\Gamma}f<\sigma<\sup_\Gamma f$.
 	Suppose that $(f,\Gamma)$ satisfies \eqref{concave}, \eqref{elliptic-weak}, \eqref{key-123}, $\Gamma\neq\Gamma_n$
 	and 
 	\begin{equation}
 		\label{unbound-2}
 		\begin{aligned}
 			f(t,\cdots,t, 0)>\sigma  \mbox{ for some } t>0. 
 		\end{aligned}
 	\end{equation} 
 	Then $\tilde{f}$ is of fully uniform ellipticity when  restricted to  $\{\lambda\in \tilde{\Gamma }: \tilde{f}(\lambda)=\sigma\}$.
 \end{corollary}
 
 Corollary \ref{coro-n-1-equ} immediately implies the following:
 \begin{proposition} 
 	If \eqref{concave}, \eqref{elliptic-weak}, \eqref{neqgamma_n}, \eqref{con10} and \eqref{con11} hold, then equation \eqref{n-1-equ1} is of fully uniform  ellipticity at any solution %$u\in C^2(B_r)$ 
 	satisfying \eqref{admissible-f1}. 
 \end{proposition} 
 This proposition confirms all the assumptions imposed in Theorem  \ref{thm1-inter-3}, 
 %and \ref{thm1-inter-4},
  thereby obtaining Theorem \ref{thm1-inter}.

 \section{Applications to Hessian
 	% and Weingarten
 	  equations
  }
 \label{sec5}

 Let $(M,g)$ be a $n$-dimensional compact Riemannian manifold, possibly with boundary $\partial M$, $\bar M=M\cup \partial M$.
 We consider
 %Weingarten equation and
the  Hessian equations 
 \begin{equation}
 	\label{hessianequ1-riemann}
 	\begin{aligned}
 		f(\lambda(\nabla^2 u+A))=\psi.
 	\end{aligned}
 \end{equation}
where $\psi$ is a smooth function and $A$ is a smooth symmetric $(0,2)$-type tensor.
 The second order estimate for \eqref{hessianequ1-riemann} on curved Riemannian manifolds was studied by 
 Guan \cite[Section 3]{Guan12a},
 % and Sz\'ekelyhidi \cite[Section 8]{Gabor},
 extending previous results in literature, see e.g. \cite{LiYY1990,Urbas2002}.
 The gradient estimate is the remaining task to the study of Hessian equations.
 However, it is rather hard to prove gradient estimate for Hessian equations on curved Riemannian manifolds.
 The gradient estimate was obtained in \cite{LiYY1990} under  assumptions \eqref{t1-to-0}, 
  $\lim_{|\lambda|\rightarrow +\infty}\sum_{i=1}^n f_i(\lambda)=+\infty$ and that the Riemannian manifold admits nonnegative sectional curvature, and later extended by Urbas \cite{Urbas2002} 
 %followed by \cite{Guan12a}, 
 with replacing such restrictions by \eqref{key2-yuan} and \eqref{positive-1}. 
One may use  Lemmas \ref{yuan-lemma1-weingarten} and \ref{yuan-lemma2-weingarten}
 to improve their results.
 
 In fact, Theorem \ref{coro1.6} and Proposition \ref{mainthm-2} allow us to derive the 
 %(local and global)
 gradient estimate for more general equations  
% According to
 % \eqref{345} in  Proposition \ref{mainthm-2},
 %\eqref{key2-yuan-2} in Theorem \ref{coro1.6}, if  $f$ in addition   
 when  
 \begin{equation}
 	\label{key3-yuan}
 	\sum_{i=1}^n f_i(\lambda)\lambda_i\geq -K_0\sum_{i=1}^n f_i(\lambda) 
	%\mbox{ in } \Gamma^{\underline{\psi},\overline{\psi}}   
	\mbox{ for some } K_0\geq0.
 \end{equation}
%for some $K_0\geq0$.
  If \eqref{key3-yuan} holds for any $\lambda\in \Gamma^{\underline{\psi},\overline{\psi}}$, then according to
  %\eqref{345} in  
  Proposition \ref{mainthm-2} and \eqref{sumfi-2}, 
 %\eqref{key2-yuan-2} in 
 %and Theorem \ref{coro1.6},
  we obtain
 a more general inequality than \eqref{key2-yuan} 
 \begin{equation}
 	\begin{aligned}
 		f_i(\lambda) \geq \theta+\theta \sum_{j=1}^n f_i(\lambda)  \quad\mbox{if } \lambda_i\leq -K_0, \mbox{  } \forall  \lambda\in \Gamma^{\underline{\psi},\overline{\psi}}. 
 		%\inf_M\psi\leq f(\lambda)\leq \sup_M\psi. \nonumber	
 	\end{aligned}
 \end{equation}
 %where $\Gamma^{\underline{\psi},\overline{\psi}}$ is defined as in \eqref{con11-2}.  
 As an application, one can follow a strategy, analogous to that used in \cite{Urbas2002}, to 
 derive gradient bound for solutions to  \eqref{hessianequ1-riemann}  under the  $\mathcal{C}$-subsolution assumption 
 \begin{equation}
 	\label{existenceofsubsolution2}
 	\begin{aligned}
 		\lim_{t\rightarrow +\infty}f(\lambda(\nabla^2\underline{u}+A)+t\mathrm{e}_i)>\psi 
 		\mbox{ in } \bar M, \quad \forall 1\leq i\leq n   %\nonumber
 	\end{aligned}
 \end{equation}
 where 
 $\mathrm{e}_i$ is the $i$-$\mathrm{th}$ standard basis vector of $\mathbb{R}^n$.
 % Following   \cite{Gabor} we call a $C^2$-smooth function $\underline{u}$ a $\mathcal{C}$-subsolution of \eqref{hessianequ1-riemann} if it satisfies \eqref{existenceofsubsolution2}. 
 %Following \cite{CNS3} $u$ is admissible for \eqref{hessianequ1-riemann} if $\lambda(\nabla^2 u+A)\in \Gamma$
 
 \begin{proposition}
 	\label{thm1-gradient}
 	%Let $(M,g)$ be a compact Riemannian manifold possibly with boundary.
 	In addition to \eqref{concave},  \eqref{elliptic-weak} and \eqref{key3-yuan}, %and \eqref{addistruc}, 
 	we assume there is a $C^2$-smooth $\mathcal{C}$-subsolution $\underline{u}$. %to the equation.
 	Let   $u\in C^3(M)\cap C^1(\bar M)$ be a  solution to \eqref{hessianequ1-riemann} with $\lambda(\nabla^2 u+A)\in \Gamma$, %$\sup_{\partial\Gamma}f<\psi <\sup_\Gamma f$
 	then 
 	\begin{equation}
 		\begin{aligned}
 			\sup_{M}|\nabla u|\leq C(1+\sup_{\partial M}|\nabla u|), \nonumber
 		\end{aligned}
 	\end{equation}
 	where $C$ depends on $|\psi|_{C^1(\bar M)}$, $|\underline{u}|_{C^2(\bar M)}$ and other known data.  
 	%Moreover, $C$ is independent  $(\delta_{\psi,f})^{-1}$.
 \end{proposition}

 With gradient estimate at hand, as in \cite{Guan12a,yuan2021cvpde}, we can prove the following:
 
 \begin{theorem}
 	
 	Let $(M^n,g)$ be a compact Riemannian manifold with smooth smooth boundary.
 	Let  $\varphi\in C^\infty(\partial M)$ and $\psi\in C^\infty(\bar M)$ be a function satisfying $\inf_M\psi>\sup_{\partial\Gamma}f$. 
 	Suppose that there is an admissible function 
 	$\underline{u}\in C^{3,1}(\bar M)$ satisfying $$f(\lambda(\nabla^2 \underline{u}+A))\geq \psi, \quad  \underline{u}|_{\partial M }=\varphi.$$ In addition to \eqref{elliptic}, \eqref{concave}, we assume  \eqref{key3-yuan} holds for $$\inf_M\psi\leq f(\lambda)\leq \sup_M f(\lambda(\nabla^2 \underline{u}+A)).$$ 
 	%\begin{equation}
 	% f(\nabla^2 u+A)\geq \psi \mbox{ in } M, \quad \underline{u}=\varphi \mbox{ on }  \partial M.	 \nonumber
 	%	\end{equation}
 	Then there exists a unique smooth admissible function $u$ solving \eqref{hessianequ1-riemann} with $u|_{\partial  M}=\varphi$.
 \end{theorem}

 \begin{remark}
 	Condition \eqref{key3-yuan} for  $K_0=0$ was also used by \cite{Guan12a} to derive boundary estimate for Dirichlet problem of 
 	%Hessian type equation
 	\eqref{hessianequ1-riemann}, and later   by \cite{GSS14} (for $K_0\geq0$) to %derive boundary estimate for
 	study first initial boundary problems.  
 	Our results %(Proposition \ref{mainthm-2}) %on partial uniform ellipticity
 	indicate that the technique assumptions $\mathrm{\bf (i)}$-$\mathrm{\bf (iii)}$ imposed in \cite[Theorem 1.10]{Guan12a}  as well as  assumptions $\mathrm{\bf (i)}$-$\mathrm{\bf (iv)}$ in \cite[Theorem 1.9]{GSS14} can be removed.
 \end{remark}

 % \begin{remark}
 
 %	When $(M,g)$ is a closed Riemannian manifold and $f$ satisfies \eqref{addistruc}, Proposition \ref{thm1-gradient} was stated in \cite[Section 8]{Gabor}, based on a blow-up argument. 	Our result  slightly extends those in \cite[Section 8]{Gabor}. 
 %since 	\eqref{key3-yuan} is broader than \eqref{addistruc}	according to Lemma \ref{lemma-new-2}. 
 %\end{remark} 

  \subsection*{Acknowledgements}
 The author was supported by the National Natural Science  Foundation of China under grant 11801587.
 
% \bigskip
 
 \begin{appendix}
 	%\medskip
 	
 	\section{Some standard lemmas}
 In this appendix we summarize some standard lemmas.
 	% (see for example \cite{Gabor} for Lemma \ref{k-buchong1}, and see \cite{GQY2018} for Lemma \ref{k-buchong2}).
 	
 	\begin{lemma}
 		\label{k-buchong2}
 		Let $f$ satisfy \eqref{concave} and \eqref{elliptic-weak}, then  \eqref{sumfi-02} holds.
		%for any $\lambda\in\Gamma$ with $f(\lambda)<\sup_\Gamma f$, 
 		%\[\sum_{i=1}^n f_i(\lambda)>0, \quad \forall %\lambda\in\Gamma \mbox{ with }
 		%f(\lambda)<\sup_\Gamma f, \mbox{ } \lambda\in\Gamma.\]
 	\end{lemma}
The above lemma has been used in \cite{GQY2018}. 
 	%Next we present the proof.
 	Below we present another one.
 	% which was previously used  in \cite{Gabor,yuan2021cvpde}.
 	
 	%\begin{proof}
 	%[Proof of Lemma \ref{k-buchong2}]
 	
 	%Suppose %$\sum_{i=1}^n f_i(\lambda)=0$. Then
 	% $f_i(\lambda)=0$ $\forall i$. Then $f(\lambda)\geq f(\mu), \quad \forall\mu\in\Gamma.$ It is a contradiction.
 	
 	%\end{proof}

 	\begin{lemma}
 		\label{k-buchong1}
 		Suppose $f$ satisfies \eqref{concave} and \eqref{elliptic-weak}.
 		Then for any $\sigma$
 		with $\sigma<\sup_\Gamma f$ and $\partial\Gamma^\sigma\neq \emptyset$,
 		there exists $c_\sigma \vec{\bf 1}\in \partial\Gamma^\sigma.$ 
 		%Here $\partial\Gamma^\sigma=\{\lambda\in\Gamma: f(\lambda)=\sigma\}$ and $\vec{\bf 1}=(1,\cdots,1)\in \mathbb{R}^n.$
 	\end{lemma}

 	\begin{proof}
 		
 		%Fix $\sigma$. 
 		%For $\lambda\in\partial\Gamma^\sigma$ with $\sigma<\sup_\Gamma f$, 
 		%$\sum_{i=1}^n f_i(\lambda)>0.$
 		For $\sigma<\sup_\Gamma f$, the level set $\partial\Gamma^\sigma$ (if $\partial\Gamma^\sigma\neq \emptyset$) is  a %smooth
 		convex noncompact %complete
 		hypersurface contained in $\Gamma$. Moreover, $\partial\Gamma^\sigma$ is symmetric with respect to the diagonal.
 		
 		Let $\lambda^0\in \partial\Gamma^\sigma$ be the closest point to the origin. (Such a point exists, since $\partial\Gamma^\sigma$ is a closed set). The idea is to prove $\lambda^0$ is %contained in the diagonal.
 		the one what we look for.
		
 		Assume $\lambda^0$ is not in the diagonal. Then by the Implicit Function Theorem, and the convexity and symmetry of $\partial\Gamma^\sigma$, one can choose $\lambda\in \partial\Gamma^\sigma$
 		which has strictly less distance than that of $\lambda^0$. It is a contradiction.
 		
 	\end{proof}

 	\section{Criterion for %concave symmetric functions
 	$f$ satisfying \eqref{addistruc}}
 	\label{appendix1}

 	We summarize characterizations of $f$ when it satisfies \eqref{concave} and \eqref{addistruc}. 
 	The following lemma was first proposed by 
 	\cite{yuan2021-2}\renewcommand{\thefootnote}{\fnsymbol{footnote}}\footnote{The paper   is extracted from
 		[arXiv:2203.03439] and the first parts of  %\cite{yuan2019,yuan2021}.
 		[arXiv:2001.09238; arXiv:2106.14837].}
 	%\cite{yuan2019}
 and further reformulated in %\cite{yuan2020conformal}.
  	\cite{yuan-PUE-conformal}.
 	\begin{lemma}[\cite{yuan2021-2,yuan-PUE-conformal}]
 		\label{lemma3.4}
 		For $f$ satisfying \eqref{concave}, the following statements are equivalent. %to each other.
 		%for any $\lambda$, $\mu\in \Gamma$, 
 		\begin{itemize}
 			\item $f$ satisfies \eqref{addistruc}.
 			
 			\item  $\sum_{i=1}^n f_i(\lambda)\mu_i>0, \mbox{ } \forall \lambda, \mbox{  } \mu\in \Gamma. $
 			%\begin{equation} \label{2.1} \begin{aligned}
 			% \sum_{i=1}^n f_i(\lambda)\mu_i>0, \mbox{ } \forall \lambda, \mbox{  } \mu\in \Gamma. %\nonumber
 			%  \end{aligned} \end{equation} 
 			\item $f(\lambda+\mu)>f(\lambda), \mbox{ }\forall \lambda, \mbox{  } \mu\in \Gamma.$
 		%	\item \begin{equation}   \begin{aligned}
 		%		f(\lambda+\mu)>f(\lambda), \quad\forall \lambda, \mbox{  } \mu\in \Gamma. %\nonumber
 			%	  \end{aligned} \end{equation} 
 		\end{itemize}
 	\end{lemma}
 %	As a corollary, we obtain
 	\begin{corollary}[\cite{yuan-PUE-conformal}]
 		\label{coro3.2}
 		%If $f$ satisfies %\eqref{elliptic},
 		Assume \eqref{concave} and \eqref{addistruc} hold. Then we have   \eqref{elliptic-weak} and 
 		%\eqref{addistruc-3->0}. 
		$\sum_{i=1}^n f_i(\lambda)>0$.
 		%\begin{itemize}
 		%	\item $f$ satisfies \eqref{elliptic-weak}.
 		
 		%	\item  $\sum_{i=1}^n f_i(\lambda)\lambda_i>0$ in $\Gamma$, %and $\sum_{i=1}^n f_i(\lambda)>0$.
 		%	i.e., $f(t\lambda)$ is a strictly increasing function in $t\in\mathbb{R}^+$.
 		%\end{itemize}
 		
 	\end{corollary}
 	
 	Inspired by the following key observation derived from \eqref{concave-1}
 	\begin{equation}
 		\label{010} 
 		\begin{aligned}
 			\mbox{For any } \lambda, \mbox{  } \mu\in\Gamma, \mbox{  }   \sum_{i=1}^n f_i(\lambda)\mu_i \geq \limsup_{t\rightarrow+\infty} f(t\mu)/t  \nonumber
 	\end{aligned}\end{equation}
 	the author \cite{yuan2021-2} introduced the following two conditions:
 	\begin{equation}
 		\label{addistruc-2}
 		\begin{aligned}
 			\mbox{For any } \lambda\in\Gamma, \mbox{  } 
 			\lim_{t\rightarrow+\infty} f(t\lambda)>-\infty,
 		\end{aligned}
 	\end{equation}
 	\begin{equation}
 		\label{addistruc-5}
 		\begin{aligned}
 			\mbox{For any } \lambda\in\Gamma, \mbox{  } 
 			\limsup_{t\rightarrow+\infty} f(t\lambda)/t\geq0.
 		\end{aligned}
 	\end{equation}
 	Obviously, it leads to
 	\begin{lemma}
 		[\cite{yuan2021-2}]
 		\label{lemma-new-1}
 		Suppose $f$ satisfies \eqref{concave} and \eqref{addistruc-5}.
 		Then \begin{equation}
 			\label{addistruc-04}
 			\begin{aligned}
 				\sum_{i=1}^n f_i(\lambda)\mu_i\geq 0  \mbox{ for any } \lambda, \mbox{ } \mu\in\Gamma.
 			\end{aligned} 
 		\end{equation}
 		In addition, $f_i(\lambda)\geq 0$ in $\Gamma$ for all $1\leq i\leq n$.
 		In particular  it satisfies %\eqref{addistruc-3}.
 			\begin{equation}
 			\label{addistruc-3}
 			\begin{aligned}
 				%  f_i(\lambda)\geq 0 \mbox{  } \forall i, \mbox{ and }
 				\sum_{i=1}^n f_i(\lambda)\lambda_i\geq0, \quad \forall \lambda\in\Gamma. 
 			\end{aligned} 
 		\end{equation} 
 		
 	\end{lemma}
 	
 	We have criteria for concave symmetric functions.
 	\begin{lemma}
 		\label{lemma-4b}
 		In the presence of \eqref{concave},  the following statements are equivalent.
		% to each other.
 		% \eqref{addistruc} $\Leftrightarrow$ \eqref{addistruc-3} $\Leftrightarrow$ \eqref{addistruc-2}.  
 		\begin{itemize}
 			
 			\item $f$ satisfies  \eqref{addistruc-2}.
 			\item $f$ satisfies  \eqref{addistruc-5}.
 			\item $f$ satisfies  \eqref{addistruc-04}.
 			\item $f$ satisfies  \eqref{addistruc-3}.
 			
 		\end{itemize}
 	\end{lemma}

 	%Motivated by Lemmas \ref{lemma3.4} and \ref{lemmaB.4}, we introduce a condition 
 	%	\begin{equation}\label{addistruc-010}	\begin{aligned}
 	%	\lim_{t\rightarrow+\infty} f(t\lambda)=\sup_\Gamma f  \mbox{ for any } \lambda\in\Gamma  \mbox{ with } f(\lambda)<\sup_\Gamma f.
 	%	\end{aligned}\end{equation}

 	%Combining with Lemma \ref{lemma3.4}, 
 We can deduce the following lemma when \eqref{sumfi>0} holds.

 	\begin{lemma}
 		%[\cite{yuan2021-2}]
 		\label{lemma-new-2}
 		Suppose  that  \eqref{concave} and \eqref{sumfi>0} hold.
 		%that $f$ is not a constant in $\Gamma$.   
 		Then the following statements are equivalent to each other.
 		% \eqref{addistruc} $\Leftrightarrow$ \eqref{addistruc-3} $\Leftrightarrow$ \eqref{addistruc-2}.  
 		\begin{itemize}
 			\item $f$ satisfies \eqref{addistruc}.
 			\item $f$ satisfies \eqref{addistruc-2}.
 			\item $f$ satisfies \eqref{addistruc-5}.
 			\item $f$ satisfies \eqref{addistruc-04}.
 			\item $f$ satisfies \eqref{addistruc-3}.
 		\end{itemize}
 	\end{lemma}

 	\begin{proof}

 		From Lemmas \ref{lemma3.4} and \ref{lemma-4b},
 		it requires only to prove 
 		\[\eqref{addistruc-5}  \Rightarrow  \eqref{addistruc}.\]
 		% We first notice that
 		% $$\left\{\lambda\in\Gamma:  f(\lambda)<\sup_\Gamma f \right\}\neq\emptyset.$$
 		%	For any $\sigma<\sup_\Gamma f$, there exists $\lambda\in\Gamma$ such that $f(\lambda)>\sigma$ (otherwise $\sup_\Gamma f\leq\sigma$).
 		Fix $\lambda$, $\mu\in\Gamma$.
 		% by \eqref{concave-1}
 		%\[f(t\mu)-f(\lambda) \geq \sum_{i=1}^n f_i(t\mu)(t\mu_i-\lambda_i), \mbox{ } \forall t>0.\]
 		Note that $t\mu-\lambda- \vec{\bf 1} \in\Gamma$ for $t> t_{\mu,(\lambda+{\vec{\bf 1}})}>0$. 
		Together with \eqref{concave-1} and  \eqref{sumfi>0}, 
 		we derive
 		\[f(t\mu)\geq f(\lambda) + \sum_{i=1}^n f_i(t\mu)>f(\lambda)\mbox{ for } t>t_{\mu,(\lambda+{\vec{\bf 1}})}.\] 
 		%This completes the proof.
 		
 	\end{proof}

 	%Before giving the proof we summarize the following lemma.
 	
 	%\begin{proof}[Proof of Lemma \ref{lemma-new-2}]
 	%Obviously $\eqref{addistruc-3} \Rightarrow \eqref{addistruc-2}$; $\eqref{addistruc-2} \Rightarrow \eqref{addistruc-5}$. 
 	%From Lemma \ref{lemma3.4},   \eqref{addistruc} $\Rightarrow$ \eqref{addistruc-3}.
 	
 	%It only requires to prove \eqref{addistruc-5} $\Rightarrow$ \eqref{addistruc}. Fix $\lambda$, $\mu\in\Gamma$. Since $\Gamma$ is open, $\mu-\delta_\mu \vec{\bf1}\in\Gamma$ for some $\delta_\mu>0$. By Lemma \ref{lemma-new-1}, $\sum_{i=1}^n f_i(\lambda)\mu_i \geq \delta_\mu\sum_{i=1}^n f_i(\lambda)$.
 	%Thus $\sum_{i=1}^n f_i(\lambda)\mu_i>0$ as required by  \eqref{elliptic-weak} and Lemma \ref{lemma3.4}.
 	%By Lemma \ref{lemma3.4}, $f$ satisfies \eqref{addistruc}.
 	%\end{proof}

 	\begin{lemma}
 		[\cite{yuan2021-2}]
 		\label{lemma1-con-addi}
 		If $f$ satisfies \eqref{concave}, \eqref{elliptic} and  \eqref{t1-to-0}, then it obeys \eqref{addistruc}.
 	\end{lemma}
 	
 	%\begin{proof}
 		
 		%According to Lemma \ref{k-buchong1}, \eqref{t1-to-0} implies 
 		%that \[\inf_\Gamma f>-\infty \]	which is exactly a special case of 
 		%\eqref{addistruc-2}. Then we can deduce \eqref{addistruc}   by Lemma \ref{lemma-new-2}.
 	  % $f$ satisfies \eqref{addistruc}.
 		%	\end{proof}
 		 
 			\begin{remark}
 			Lemma \ref{lemma-new-2} was proved in \cite{yuan2021-2}
 			when $f$ satisfies  \eqref{concave}-\eqref{elliptic}.
 		\end{remark}

 \end{appendix}

 % \noindent
 % {\bf  Acknowledgement.} The author is supported by the National Natural Science Foundation of China (Grant No. 11801587).

 % \bigskip
 
 %  \subsection*{Acknowledgement} The author is supported by the National Natural Science Foundation of China (Grant No. 11801587).

\end{document}